\documentclass[12pt,reqno]{amsproc}

\usepackage[margin=1in]{geometry}

\usepackage[textsize=small]{todonotes}
\usepackage{mathpazo} 
\usepackage[all,cmtip]{xy}
\usepackage{tikz-cd}
\usepackage{svg}
\usepackage[utf8]{inputenc} 
\usepackage{microtype} 
\usepackage{amsthm, amsmath, amssymb, mathtools} 
\usepackage{float} 
\usepackage[final, colorlinks = true, 
            linkcolor = blue, 
            citecolor = blue]{hyperref} 
\usepackage{graphicx, multicol} 
\usepackage{xcolor} 
\usepackage{enumitem}
\usepackage{mathrsfs}
\setlist{noitemsep}

\newtheorem{thm}{Theorem}[section]
\newtheorem*{theorem*}{Theorem}
\newtheorem{prop}[thm]{Proposition}
\newtheorem{lemma}[thm]{Lemma}
\theoremstyle{definition}
\newtheorem{example}[thm]{Example}
\newtheorem{defn}[thm]{Definition}
\theoremstyle{remark}
\newtheorem{rem}[thm]{Remark}

\DeclareMathOperator{\G}{{\text{\rm G}}}
\DeclareMathOperator{\HH}{{\text{\rm H}}}
\DeclareMathOperator{\K}{{\text{\rm K}}}
\newcommand{\F}[1]{\mathbf{F}{#1}}
\newcommand{\M}[1]{\mathbf{M}{#1}}
\DeclareMathOperator{\h}{\mathbf{H}}
\DeclareMathOperator{\E}{\mathbf{E}}
\DeclareMathOperator{\LL}{\mathbf{L}}
\DeclareMathOperator{\p}{\mathbf{P}}
\DeclareMathOperator{\grp}{\mathbf{Grp}}
\DeclareMathOperator{\cset}{\mathbf{c}}
\DeclareMathOperator{\C}{\mathcal{C}}
\DeclareMathOperator{\Q}{\mathcal{Q}}
\DeclareMathOperator{\D}{\mathcal{D}}

\DeclareMathOperator{\cst}{\underline{\G}}

\newcommand{\ent}[1]{\mathbf{Ent}[#1]}
\newcommand{\fac}[1]{\mathbf{Fac}[#1]}
\newcommand{\loc}[1]{\mathbf{Loc}_\Sigma\ent{#1}}
\newcommand{\flo}[1]{\mathbf{Flo}_\Sigma[{#1}]}

\newcommand{\ccog}[2]{\underline{{#2}}_{{#1}}}
\newcommand{\set}[1]{\left\{{#1}\right\}}
\definecolor{emerald}{HTML}{093145}
\definecolor{flame}{HTML}{829356}
\definecolor{aubergine}{HTML}{107896}

\renewcommand{\mathcal}{\mathscr}

\newcommand{\inj}{\hookrightarrow}
\newcommand{\surj}{\twoheadrightarrow}
\newcommand{\two}{\Rightarrow}
\newcommand{\btrt}{\blacktriangleright}

\usepackage{siunitx}
\title{Morse Theory for Complexes of Groups}
\author{Naya Yerolemou and Vidit Nanda}
\date{}
\begin{document}
	

	\begin{abstract} We construct an equivariant version of discrete Morse theory for simplicial complexes endowed with group actions. The key ingredient is a 2-categorical criterion for making acyclic partial matchings on the quotient space compatible with an overlaid complex of groups. We use the discrete flow category of any such compatible matching to build the corresponding Morse complex of groups. Our main result establishes that the development of the Morse complex of groups recovers the original simplicial complex up to equivariant homotopy equivalence.
	\end{abstract}
	
	\maketitle

	\section{Introduction}\label{sec:intro}

	The interplay of ideas between group actions and Morse theory has resulted in several substantial advances across modern geometry and topology. Any highlight reel featuring equivariant Morse theory would certainly include Atiyah and Bott's seminal study of algebraic bundles on Riemann surfaces \cite{atiyahbott}, Hingston's work on counting closed geodesics in symmetric spaces \cite{hingston}, and Kirwan's investigations of symplectic quotients \cite{kirwan}, among other results. Such efforts are motivated (and hindered) by the fact that -- unless a group $\G$ acts freely on a topological space $X$ -- the quotient $Y := X/{\G}$ is usually far more pathological than $X$ itself. For instance, $Y$ may be singular even when $\G$ is a Lie group and $X$ a smooth manifold. Similarly, the quotients of algebraic varieties by algebraic group actions typically fail to be varieties in their own right. In such cases, it becomes necessary to augment $Y$ with additional stabiliser data to form a {\em quotient stack}. 
	
	Our goal in this paper is to produce an avatar of equivariant Morse theory specially adapted to quotient stacks arising from actions of finite groups $\G$ on finite simplicial complexes $X$. Under mild hypotheses on such actions (called {\em regularity}), one can ensure that the quotient $Y$ is also a simplicial complex and that the canonical projection $X \surj Y$ is a surjective simplicial map even when the $\G$-action is not free. The main instruments of this paper are a pair of 2-functors, one arising from the quotient stack itself and the other from a suitable discretisation of Morse theory. 
	
	\subsection*{Complexes of Groups} Quotient stacks for group actions on simplicial complexes are called complexes of groups, and their study dates back to the work of Corson \cite{corson} and Haefliger \cite{haefliger} in geometric group theory. A complex of groups associated to the action of $\G$ on $X$ amounts to a 2-functor of the form
	\[
	\F{}:\fac{Y} \to \grp.
	\]
	Here the domain is the poset of simplices (in the quotient complex $Y = X/{\G}$) ordered by the co-face relation, while the codomain $\grp$ is the 2-category whose objects are all groups, whose 1-morphisms are ordinary group homomorphisms, and the 2-morphism structure arises from conjugation \cite{carbone2020equivariant}. This is the first 2-functor of interest to us in this paper.
	
	When $\G$ is abelian, there is a considerable simplification since all 2-morphisms in the image of $\F{}$ become identities. For the purposes of these introductory remarks, we employ the following running example: consider the simplicial complex $X$ drawn below
	\begin{center}
	    \includegraphics[width=0.5\textwidth]{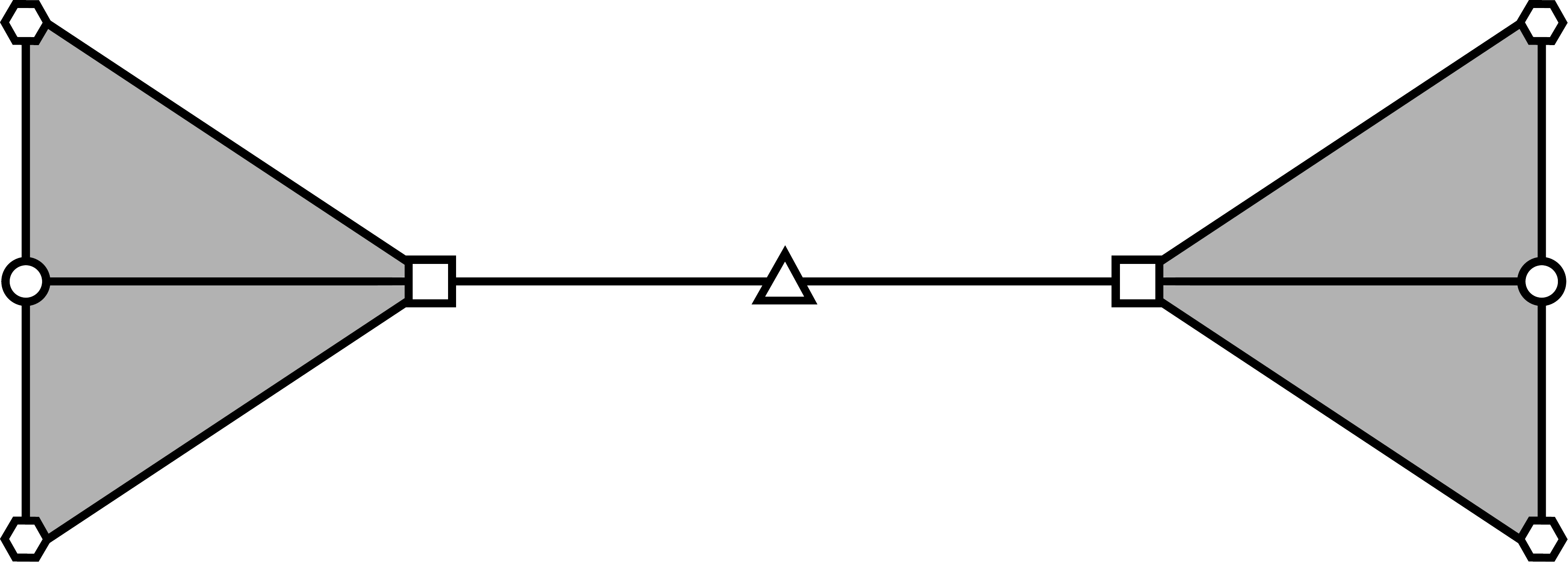}
	\end{center}
	The Klein four-group  $\G = \langle\sigma,\tau \mid \sigma^2=1=\tau^2\rangle$ acts on $X$ by letting $\sigma$ and $\tau$ induce reflections about the horizontal and vertical axis through the central vertex $\Delta$, and the resulting quotient $Y$ is:
	\begin{center}
	    \includegraphics[width=0.3\textwidth]{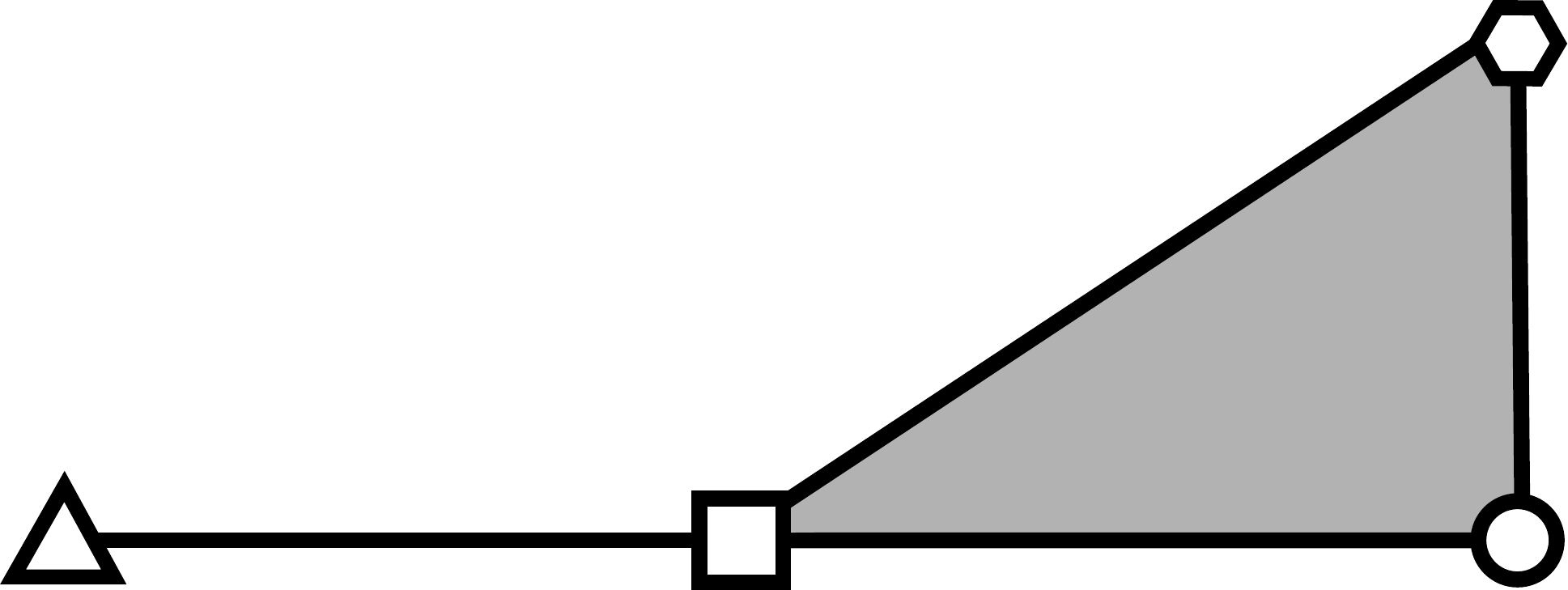}
	\end{center}
	
	A complex of groups $\F{}$ for this action is illustrated below (all of the 2-morphisms in sight are trivial):
    \begin{center}
        \fontsize{8pt}{8pt}
        \def\svgwidth{0.6\columnwidth}
        \scalebox{1}{
\begingroup%
  \makeatletter%
  \providecommand\color[2][]{%
    \errmessage{(Inkscape) Color is used for the text in Inkscape, but the package 'color.sty' is not loaded}%
    \renewcommand\color[2][]{}%
  }%
  \providecommand\transparent[1]{%
    \errmessage{(Inkscape) Transparency is used (non-zero) for the text in Inkscape, but the package 'transparent.sty' is not loaded}%
    \renewcommand\transparent[1]{}%
  }%
  \providecommand\rotatebox[2]{#2}%
  \newcommand*\fsize{\dimexpr\f@size pt\relax}%
  \newcommand*\lineheight[1]{\fontsize{\fsize}{#1\fsize}\selectfont}%
  \ifx\svgwidth\undefined%
    \setlength{\unitlength}{1169.68625732bp}%
    \ifx\svgscale\undefined%
      \relax%
    \else%
      \setlength{\unitlength}{\unitlength * \real{\svgscale}}%
    \fi%
  \else%
    \setlength{\unitlength}{\svgwidth}%
  \fi%
  \global\let\svgwidth\undefined%
  \global\let\svgscale\undefined%
  \makeatother%
  \begin{picture}(1,0.53070291)%
    \lineheight{1}%
    \setlength\tabcolsep{0pt}%
    \put(0,0){\includegraphics[width=\unitlength,page=1]{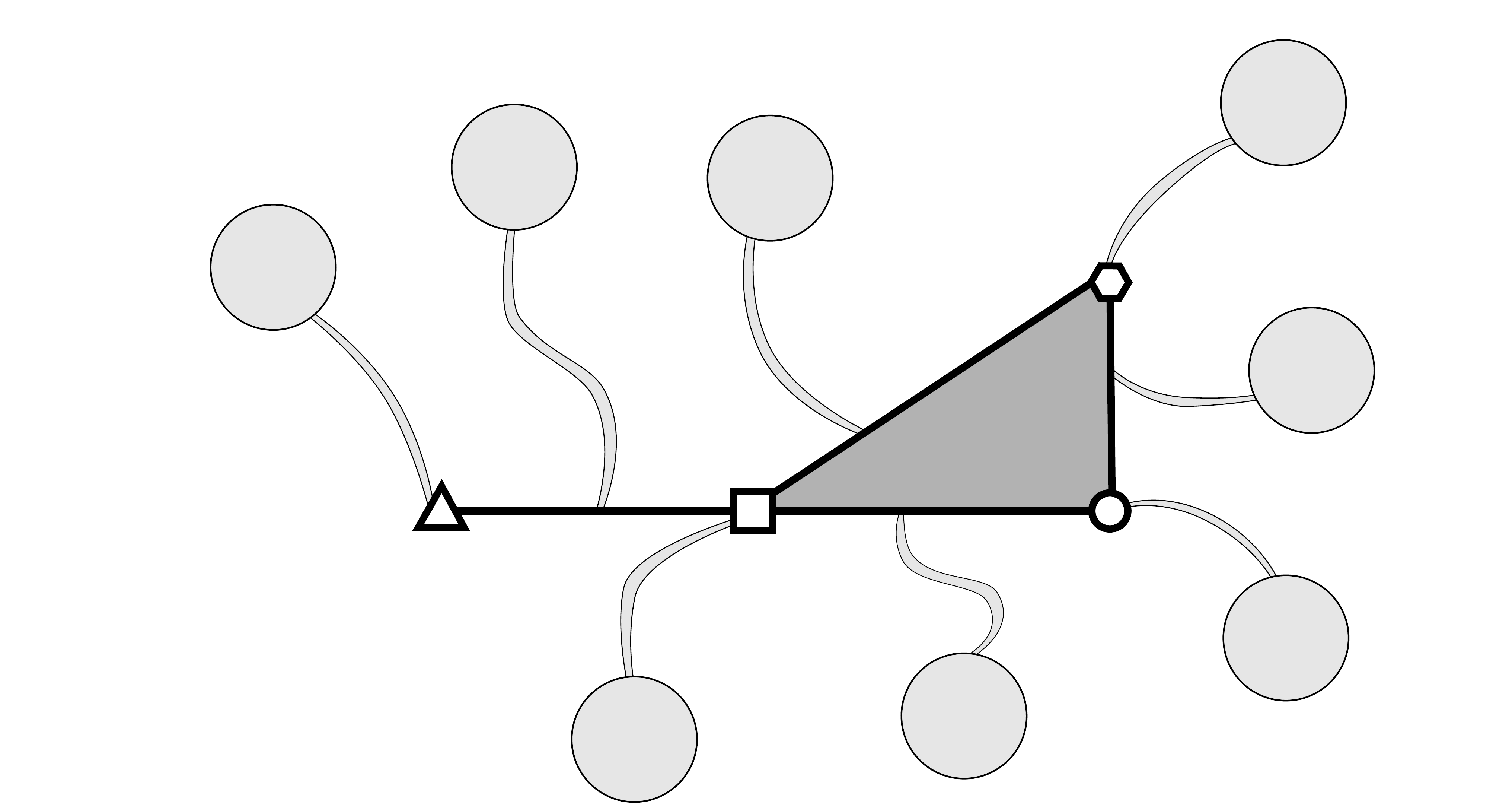}}%
    \put(0.18067438,0.34676905){\makebox(0,0)[t]{\lineheight{1.25}\smash{\begin{tabular}[t]{c}$\langle\sigma,\tau\rangle$\end{tabular}}}}%
    \put(0.34011997,0.41298914){\makebox(0,0)[t]{\lineheight{1.25}\smash{\begin{tabular}[t]{c}$\langle\sigma\rangle$\end{tabular}}}}%
    \put(0.41952597,0.03474567){\makebox(0,0)[t]{\lineheight{1.25}\smash{\begin{tabular}[t]{c}$\langle\sigma\rangle$\end{tabular}}}}%
    \put(0.63762465,0.05020136){\makebox(0,0)[t]{\lineheight{1.25}\smash{\begin{tabular}[t]{c}$\langle\sigma\rangle$\end{tabular}}}}%
    \put(0.85057138,0.10172071){\makebox(0,0)[t]{\lineheight{1.25}\smash{\begin{tabular}[t]{c}$\langle\sigma\rangle$\end{tabular}}}}%
    \put(0.86758275,0.27868613){\makebox(0,0)[t]{\lineheight{1.25}\smash{\begin{tabular}[t]{c}$\langle 1\rangle$\end{tabular}}}}%
    \put(0.84889315,0.45555999){\makebox(0,0)[t]{\lineheight{1.25}\smash{\begin{tabular}[t]{c}$\langle 1\rangle$\end{tabular}}}}%
    \put(0.50936487,0.40572108){\makebox(0,0)[t]{\lineheight{1.25}\smash{\begin{tabular}[t]{c}$\langle 1\rangle$\end{tabular}}}}%
    \put(0,0){\includegraphics[width=\unitlength,page=2]{intro_cplx.pdf}}%
    \put(0.65472863,0.48151787){\makebox(0,0)[t]{\lineheight{1.25}\smash{\begin{tabular}[t]{c}$\langle 1\rangle$\end{tabular}}}}%
  \end{picture}%
\endgroup%
}
    \end{center}
	Note that the group assigned to each simplex of $Y$ is a subgroup of the groups assigned to its faces. Similarly, the trivial $\G$-action gives rise to the {\em constant} complex of groups $\ccog{Y}{\G}$, where every simplex of $Y$ is assigned $\G$, all face relations are assigned identity group homomorphisms, and all 2-morphisms are trivial. Our complex of groups $\F{}$ admits a natural embedding (i.e., an injective natural transformation) $\Phi:\F{} \inj \ccog{Y}{\G}$ to the constant complex of groups. Associated to  this pair $(\F{},\Phi)$ is a category $\D(\F{},\Phi)$, which carries a natural $\G$-action and is called the {\em development} of $(\F{},\Phi)$. The following result, paraphrased from \cite{bridson2011metric}, is called {\bf the basic construction}.
	
	\begin{theorem*}[A] Let $(\F{},\Phi)$ be the complex of groups and injective natural transformation associated to the action of a group $\G$ on a simplicial complex $X$. Then the nerve of its development $\D(\F{},\Phi)$ is $\G$-equivariantly isomorphic to $X$.
	\end{theorem*}
	
	In fact, the basic construction is far more general than the above description suggests --- one can consider the action of $\G$ on an arbitrary loopfree category $\C$ rather than restricting focus to the poset of simplices in a given simplicial complex $X$. In this case, the associated development recovers $\C$ up to $\G$-equivariant isomorphism.
	
	\subsection*{Discrete Morse Theory} The second 2-functor of interest to us here arises from a  combinatorial adaptation of Morse theory due to Forman \cite{forman1998morse}. The central objects of study here are simplicial (or more generally, regular CW) complexes rather than smooth manifolds. The role of the gradient vector field is played by a partial bijection $\Sigma$ relating adjacent simplices of co-dimension one, subject to a global acyclicity condition. These pairings are typically illustrated as arrows from the lower-dimensional simplex to the higher-dimensional one. Here are two such {\em acyclic partial matchings} on our running example $X$ --- the one on the left is $\G$-equivariant whereas the one on the right is not.
	\begin{center}
	    \includegraphics[width=0.95\textwidth]{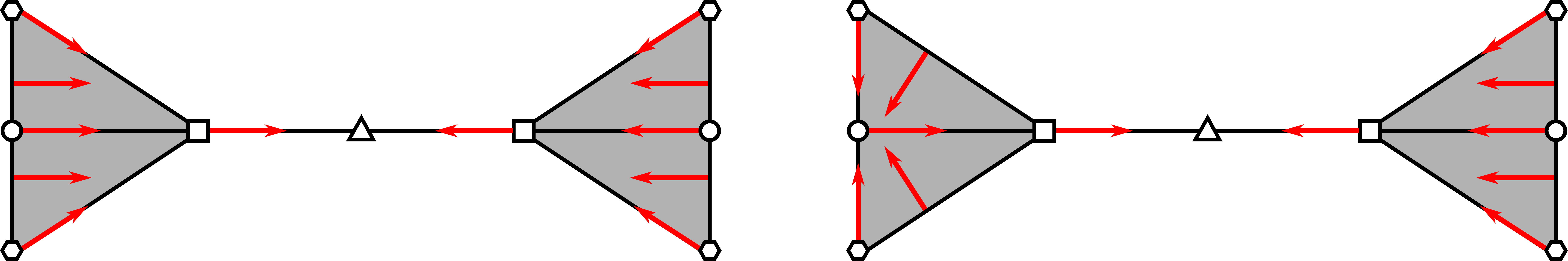}
	\end{center}
	
	The unpaired simplices are analogous to critical points from smooth Morse theory in the sense that the homotopy type of $X$ may be recovered from a CW complex whose $k$-cells correspond bijectively with $k$-dimensional unpaired simplices. In the two acyclic partial matchings on $X$ depicted above, only the central vertex $\Delta$ remains unpaired. Thus, both matchings establish that $X$ is contractible; but the one on the left goes a step further by showing that in fact $X$ is $\G$-equivariantly contractible. Indeed, one may view the equivariant matching as a lift of the following acyclic partial matching on the quotient $Y$:
	\begin{center}
	\includegraphics[width=0.25\textwidth]{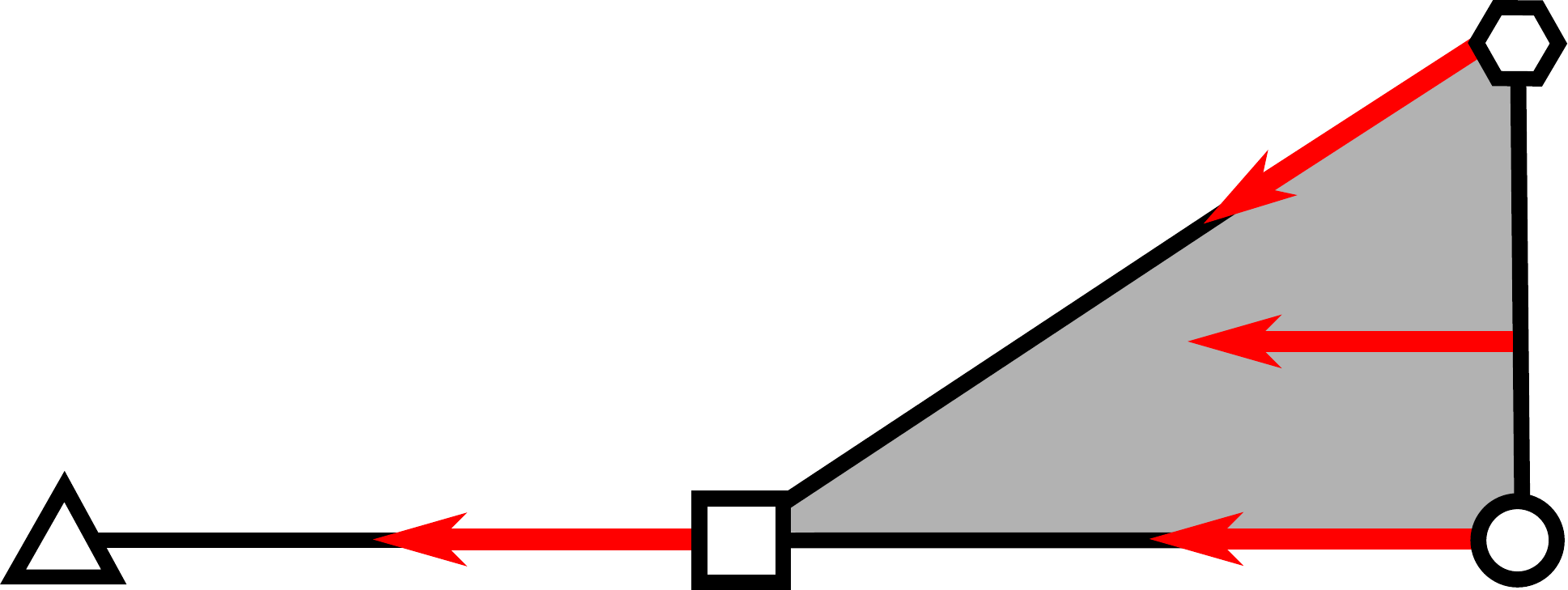}
	\end{center}
	
	Unfortunately, the CW complex induced by unpaired simplices is built by induction and its attaching maps are not straightforward to describe in general. To get a better handle on the homotopy type of this complex, we turn to the {\em entrance path category} $\ent{Y}$ of the quotient simplicial complex $Y$. This is the 2-category whose objects are the simplices of $Y$ and 1-morphisms $y \to y'$  consist of all strictly descending sequences of faces
	\[
	y_0 > y_1 > \cdots >y_{k-1} > y_k
	\]
	with $y=y_0$ and $y_k=y'$; the 2-morphisms arise from partially ordering such sequences by inclusion. Every acyclic partial matching $\Sigma$ corresponds to a set of minimal entrance paths, and the key to recovering the homotopy type of $Y$ is to localise $\ent{Y}$ by formally inverting all the 1-morphisms implicated by $\Sigma$. The second 2-functor of interest to us is
	\[
	\LL_\Sigma:\ent{Y} \to \loc{Y},
	\] i.e., the canonical  functor from the entrance path category to its localisation about $\Sigma$. Here is the main result of \cite{NANDA2019459}. 
	
	\begin{theorem*}[B] If $\Sigma$ is an acyclic partial matching on a simplicial complex $X$, then
	\begin{enumerate}
	    \item the localisation functor $\LL_\Sigma$ induces a homotopy equivalence of classifying spaces; and
	    \item writing $\flo{Y}$ for the full subcategory of $\loc{Y}$ spanned by unpaired simplices, the inclusion $\flo{Y} \inj \loc{\Sigma}$ also induces a homotopy equivalence.
	\end{enumerate}
	\end{theorem*}
The classifying space of $\ent{Y}$ is homotopy equivalent to $Y$, so the above result allows us to recover $Y$ up to homotopy type from the {\bf flow category} $\flo{Y}$. This category has the unpaired simplices as objects and (posets of) $\Sigma$-localised entrance paths as morphisms.
	
	\subsection*{This Paper} Here we simultaneously generalise Theorems (A) and (B) by building a discrete Morse theory for complexes of groups. The basic steps (and our main contributions) are as follows:
	\begin{enumerate}
	    \item We extend complexes of groups and the basic construction to account for group actions on loopfree {\em poset-enriched} categories.
	    \item We give a criterion for making acyclic partial matchings on the quotient $X/{\G}$ {\em compatible} with the $\G$-action on $X$ in terms of an associated complex of groups.
	    \item For any such compatible $\Sigma$, we construct a pair $(\M{},\Psi)$ where $\M{}$ is a new complex of groups on $\flo{Y}$ while $\Psi$ is an injective natural transformation to $\ccog{\flo{Y}}{\G}$.
	\end{enumerate}
	
	We call $\M{}$ a {\bf Morse complex of groups} for $\Sigma$. Here is our main result.
	
	\begin{theorem*}
		    Let $(\F{},\Phi)$ denote a complex of groups and injective natural transformation associated to the regular action of a finite group $\G$ on a finite simplicial complex $X$. Assume that $\Sigma$ is an $\F{}$-compatible acyclic partial matching on the quotient $Y = X/\G$. If  ${\M{}:\flo{Y}\to\grp}$ is an associated Morse complex of groups and  $\Psi:\M{}\inj\cst_{\flo{Y}}$ the corresponding injective natural transformation, then the classifying space of the development $\D(\M{},\Psi)$ is $\G$-equivariantly homotopy-equivalent to $X$.
	\end{theorem*}
	
	\noindent The basic strategy of our proof involves lifting $\Sigma$ to a $\G$-equivariant acyclic partial matching $\tilde{\Sigma}$ on $X$. We then inherit natural $\G$-actions on $\ent{X}$, its localisation about $\tilde{\Sigma}$, and the flow category. Finally, we establish that the development $\D(\M{},\Psi)$ is isomorphic to the flow category ${\bf Flo}_{\tilde{\Sigma}}[X]$.
	
	\subsection*{Related Work}
	Morse theories for differentiable stacks have been studied by Hepworth \cite{hepworth1, hepworth2} as well as Cho-Hong \cite{chohong}. Flow categories of standard (i.e., smooth) Morse functions date back to the work of Cohen, Jones and Segal \cite{cjs}, although we are unaware of any work relating their smooth flow category to various stacky Morse theories. On the other hand, an equivariant version of discrete Morse theory for simplicial complexes was introduced by Freij in \cite{freij}. This theory does not readily descend to quotient stacks --- in particular, the group actions considered here are not necessarily regular, and hence do not admit complexes of groups. 
	
	\subsection*{Outline} In Section \ref{sec:lpcat} we describe loopfree poset-enriched categories and their functors. Sections \ref{sec:enrichedcog}, \ref{sec:grpact} and \ref{sec:dev} are dedicated to the study of group actions on these LP-categories, building the associated complexes of groups, and the basic construction. Section \ref{sec:flow} tersely summarises the LP-categorical approach to discrete Morse theory via the flow category. Section \ref{sec:morsecog} forms the heart of this paper --- here we describe the compatibility criterion for acyclic partial matchings in the presence of group actions and the construction of Morse complexes of groups for such matchings. Finally, we prove our main result in Section \ref{sec:equhomequiv} and illustrate the entire process via an example in Section \ref{sec:example}.
	
	\subsection*{Acknowledgements} We are grateful to Andr\'{e} Henriques and Ulrike Tillmann for insightful discussions. NY is supported by The Alan Turing Institute under the EPSRC grant EP/N510129/1. VN's work is supported by EPSRC grant EP/R018472/1.

	\section{LP-Categories}\label{sec:lpcat}
	
	A small category $\C$ is said to be {\em poset-enriched} if for any pair of objects $x$ and $y$ in $\C$, the set of morphisms $\C(x,y)$ from $x$ to $y$ is endowed with a partial order $\two$ so that across any triple $x,y$ and $z$ of such objects, the map
	\[
	\C(x,y) \times \C(y,z) \stackrel{\circ}{\longrightarrow} \C(x,z)
	\]
	sending $(f,f')$ to the composite $f' \circ f$ is order-preserving. In other words, whenever we have $f_0 \two f_1$ in the poset $\C(x,y)$ and $f'_0 \two f'_1$ in the poset $\C(y,z)$, then we also have the order relation $f'_0 \circ f_0 \two f'_1 \circ f_1$ between composites in the poset $\C(x,z)$. We say that $\C$ is {\em loopfree} if two conditions hold:
	\begin{enumerate}
		\item for each object $x$, the poset $\C(x,x)$ only contains the identity morphism, and
		\item for objects $x \neq y$, if $\C(x,y)$ is nonempty  then $\C(y,x)$ must be empty.
	\end{enumerate}
	We will henceforth abbreviate loopfree poset-enriched categories as {\bf LP-categories}.  Given two LP categories $\C$ and $\D$, an {\bf LP functor} $F:\C \to \D$ is an ordinary functor with the additional requirement that for any pair of objects $x, y$ in $\C$ the induced map
	\[
	F_{xy}:\C(x,y) \to \D(Fx,Fy)
	\] is order-preserving. 
	
	\medskip
	
	\begin{defn}\label{def:catspc}
		The \textbf{geometric nerve ${\Delta\C}$}  of an LP category $\C$ is the simplicial set whose vertices are the objects in $\C$, and whose $n$-simplices spanning objects $x_0,...,x_n$ consist of morphisms $f_{ij}:x_i\to x_j$ satisfying $f_{ik}\two f_{jk}\circ f_{ij}$ for all $0\leq i\leq j\leq k\leq n$, with the convention that $f_{ii}=\text{id}$. The \textbf{classifying space} $|{\Delta\C}|$ of $\C$ is the geometric realisation of its nerve.
	\end{defn}
	
   \noindent  Every LP-functor $\C \to \D$ induces a simplicial map $\Delta\C \to \Delta\D$ and hence a continuous map on classifying spaces.

	\section{Complexes of Groups on LP Categories}\label{sec:enrichedcog}
	
	There is a natural 2-categorical extension of the category of groups, where the higher morphisms arise from conjugation. More precisely, the 2-category $\grp$ is defined as follows: 
	\begin{enumerate}
		\item its objects are all groups;
		\item the 1-morphisms $\grp(\G,\HH)$ are  group homomorphisms $\G \to \HH$; and,
		\item the 2-morphisms $\phi \two \phi'$ in $\grp(\G,\HH)$ consist of all group elements $h$ in $\HH$ that satisfy $\phi(g) = h \cdot \phi'(g) \cdot h^{-1}$ for all $g$ in $\G$.
	\end{enumerate}
	
	The 1-morphisms in $\grp$ are composed in the usual manner for group homomorphisms. On the other hand, the 2-morphisms in $\grp$ will be called {\em twisting elements}, and can be composed along two different axes. The {\em vertical} composition sends each pair $(h,h')$ on the left to their product $h \cdot h'$ in $\HH$ on the right:
	\[
	\xymatrixcolsep{1.2in}
	\xymatrix
	{
		\G \ar@/^2.5em/[r]|{\phi}_*+<1em>{\Downarrow h} \ar@{->}[r]|{\phi'}_*+<.7em>{\Downarrow h'} \ar@/_2.5em/[r]|{\phi''}  & \HH &  \G \ar@/^2.5em/[r]|{\phi}_*+<3.1em>{\Downarrow h \cdot h'} \ar@/_2.5em/[r]|{\phi''}  & \HH}
	\]
	We indicate this type of composition (contravariantly) by $h' * h = h \cdot h'$ and note that $h^{-1}$ serves as an inverse for $h$. Thus, for each pair of objects $\G,\HH$ the set $\grp(\G,\HH)$ forms a 1-category whose morphisms are all invertible (i.e., a groupoid). Conversely, the {\em horizontal} composite sends the pair $(h,k)$ on the left to the product $\psi(h) \cdot k$ in $\K$ on the right: 
	\[
	\xymatrixcolsep{1.2in}
	\xymatrix
	{
		\G \ar@/^2em/[r]|{\phi}_*+<2.1em>{\Downarrow h}  \ar@/_2em/[r]|{\phi'}  & \HH \ar@/^2em/[r]|{\psi}_*+<2.1em>{\Downarrow k}  \ar@/_2em/[r]|{\psi'} & \K & \G \ar@/^2em/[r]|{\psi \circ \phi}_*+<2.1em>{\Downarrow \psi(h) \cdot k}  \ar@/_2em/[r]|{\psi' \circ \phi'} & \K
	}
	\]
	This composition is denoted by $k \circ h = \psi(h) \cdot k$. For any configuration of objects, 1 and 2-morphisms in $\grp$ of the form
	\[
	\xymatrixcolsep{1.2in}
	\xymatrix
	{
		\G \ar@/^2.5em/[r]|{\phi}_*+<1em>{\Downarrow h} \ar@{->}[r]|{\phi'}_*+<.7em>{\Downarrow h'} \ar@/_2.5em/[r]|{\phi''}  & \HH \ar@/^2.5em/[r]|{\psi}_*+<1em>{\Downarrow k} \ar@{->}[r]|{\psi'}_*+<.7em>{\Downarrow k'} \ar@/_2.5em/[r]|{\psi''} & \K},
	\]
	the {\em interchange law} holds between the horizontal and vertical compositions; that is, we have
	\[
	(k'*k) \circ (h'*h) = (k'\circ h') * (k \circ h),
	\]
	which confirms that $\grp$ is a 2-category. Our primary interest here is in a certain class of $\grp$-valued pseudofunctors, as defined below (see \cite[Chapter III.C.2]{bridson2011metric}).
	\begin{defn} \label{def:cog}
		A {\bf complex of groups} $\F{}$ over an LP-category $\C$ assigns:		
		\begin{enumerate} 
			\item to each object $x$ of $\C$ a group $\F{x}$, 
			\item to each pair of objects $x,y$ in $\C$ a functor $\F{} = \F{}_{xy}$ from the poset $\C(x,y)$ to the groupoid $\grp(\F{x},\F{y})$ with $\F{f}:\F{x} \to \F{y}$ injective for all $f:x \to y$ in $\C$; and 
			\item to every pair $(f,f') \in \C(x,y) \times \C(y,z)$ of composable morphisms in $\C$, a twisting element $\gamma(f,f') = \gamma_{\F{}}(f,f')$ of the form $\F{f'} \circ \F{f} \two \F{(f' \circ f)}$.
		\end{enumerate}	
		These assignments are subject to the following axioms: first, given any morphism $f:x \to y$ in $\C$, both the twisting elements $\gamma(\text{id}_y,f)$ and $\gamma(f,\text{id}_x)$ equal the identity natural transformation on the functor $\F{f}$. And second, across any triple of composable morphisms in $\C$, say
		\[
		\xymatrixcolsep{.85in}
		\xymatrix{
			x_0 \ar@{->}[r]^{f_1} & x_1 \ar@{->}[r]^{f_2} & x_2 \ar@{->}[r]^{f_3} & x_3,
		}
		\] 
		the associated twisting elements satisfy the following {\em cocycle condition}:
		\[
		\gamma(f_2,f_3) \cdot \gamma(f_1,f_3\circ f_2)  = 
		\F{f_3}(\gamma(f_1,f_2)) \cdot \gamma(f_2\circ f_1,f_3)  
		\]
		in the group $\F{x_3}$.
	\end{defn}
	
	Equivalently, $\F{}: \C \to \grp$ is a pseudofunctor that sends every morphism $f:x \to y$ in $\C$ to an injective group homomorphism $\F{f}: \F{x} \inj \F{y}$ in $\grp$. The cocycle condition described above is designed to ensure that the following associativity diagram commutes in the groupoid $\grp(\F{x_0},\F{x_3})$:
	\[
	\xymatrixcolsep{1.2in}
	\xymatrixrowsep{.75in}
	\xymatrix{
		\F{f_3} \circ \F{f_2} \circ \F{f_1} \ar@{=>}[r]^-{\text{id}_{\F{f_3}} \circ \gamma(f_1,f_2)} \ar@{=>}[d]_-{\gamma(f_2,f_3) \circ \text{id}_{\F{f_1}}} & \F{f_3} \circ \F{(f_2 \circ f_1)} \ar@{=>}[d]^{\gamma(f_2 \circ f_1,f_3)}\\
		\F{(f_3 \circ f_2)} \circ \F{f_1} \ar@{=>}[r]_{\gamma(f_1,f_3 \circ f_2)} &  \F{(f_3 \circ f_2 \circ f_1)}
	}
	\]
	
	The simplest example of a complex of groups, given a fixed group $\G$ and an LP-category $\C$, is the {\bf constant} $\G$-valued complex of groups over $\C$. This is denoted $\ccog{\C}{\G}$ and determined completely by the following data: it assigns the group $\G$ to every object $x$ of $\C$, the identity homomorphism $\G \to \G$ to every morphism $f:x \to y$, and the identity twisting element $1_{\G} \in \G$ to every pair of composable morphisms. The following notion (see \cite[Chapter III.C Section 2.4]{bridson2011metric} or \cite[Definition 12]{lim:thomas:08}) provides a suitable framework for comparing two complexes of groups defined on the same LP-category.
	
	\begin{defn}\label{def:cogmap}
		A {\bf morphism} $\Phi:\F{} \to \F'{}$ of complexes of groups over the LP-category $\C$ assigns
		\begin{enumerate}
			\item to each object $x$ of $\C$ a 1-morphism $\Phi_x:\F{x} \to \F'{x}$ in $\grp$, and
			\item to each 1-morphism $f:x \to y$ of $\C$ a twisting element $\tau(f) = \tau_\Phi(f)$ in $\grp$ of the form $\Phi_y \circ \F{f} \two \F'{f}\circ \Phi_x$: 
			\[
			\xymatrixcolsep{0.6in}
			\xymatrixrowsep{0.15in}
			\xymatrix{
				\F{x} \ar@{->}[dd]_{\F{f}}  \ar@{->}[rr]^{\Phi_x} &  & \F'{x}   \ar@{->}[dd]^{\F'{f}} \\
				& \scalebox{1.15}{\rotatebox[origin=c]{30}{$\stackrel{\tau(f)}{\implies}$}} & \\
				\F{y} \ar@{->}[rr]_{\Phi_y}  & &  \F'{y}
			}
			\]
		\end{enumerate}
		so that the following two axioms hold:
		\begin{enumerate}
			\item the {\em identity} axiom requires $\tau(\text{id}_x)$ to be the identity twisting element $1_{\F'{x}}:\Phi_x \two \Phi_x$ for each object $x$ of $\C$, whereas
			\item for each pair $x \stackrel{f_1}{\longrightarrow} y \stackrel{f_2}{\longrightarrow} z$ of composable $1$-morphisms of $\C$, the {\em coherence} axiom imposes a relation
			\[
			\Phi_z(\gamma_\F{}(f_1,f_2)) \cdot \tau(f_2 \circ f_1) = \tau(f_2) \cdot \F'{f_2}(\tau(f_1)) \cdot \gamma_{\F'{}}(f_1,f_2).
			\]
			in the group $\F'{z}$ among the associated twisting elements.\label{eq:coherence}
		\end{enumerate}
	\end{defn}
	
	We call $\Phi:\F{} \to \F'{}$ an epi/mono/iso-morphism of complexes of groups if each $\Phi_x$ is an epi/mono/iso-morphism of groups. The coherence axiom described above enforces the commutativity of the following pentagon in the groupoid $\grp(\F{x},\F'{z})$:
	\[
	\xymatrixcolsep{-.1in}
	\xymatrixrowsep{0.4in}
	\xymatrix{
		& \Phi_z \circ \F{f_2} \circ \F{f_1} \ar@{=>}[rr]^-{\tau(f_2) \circ \text{id}_{\F{f_1}}} \ar@{=>}[dl]_-{\text{id}_{\Phi_z} \circ \gamma_\F{}(f_1,f_2)\hspace{1em}} & & \F'{f_2} \circ \Phi_y \circ \F{f_1} \ar@{=>}[dr]^-{\text{id}_{\F'{f_2}} \circ \tau(f_1)} & \\
		\Phi_z \circ \F{(f_2 \circ f_1)} \ar@{=>}[drr]_-{\tau(f_2 \circ f_1)\hspace{1em}}  & & & &  \F'{f_2} \circ \F'{f_1} \circ \Phi_x \ar@{=>}[dll]^-{\hspace{1em}\gamma_{\F'{}}(f_1,f_2) \circ \text{id}_{\Phi_x}} \\
		& & \F'{(f_2 \circ f_1)} \circ \Phi_x & &
	}
	\]
	
	Our main interest here will be in the somewhat simpler case of monomorphisms $\Phi:\F{}\inj \ccog{\C}{\G}$, where $\ccog{\C}{\G}$ denotes the constant complex of groups over $\C$ corresponding to some fixed group $\G$. 
	
	\begin{rem}
	    	Such a morphism is prescribed by the following data: each object $x$ of $\C$ is assigned an injective group homomorphism $\Phi_x:\F{x} \inj \G$ and each morphism $f:x \to y$ of $\C$ is assigned a twisting element $\tau_\Phi(f):\F{f} \circ \Phi_y \two \Phi_x$. We further require $\tau_\Phi(\text{id}_x) = 1_{\G}$ for each $x$, and that the relation 
	\begin{align}\label{eq:injid}
	\Phi_z(\gamma_{\F{}}(f_1,f_2)) \cdot \tau_\Phi(f_2\circ f_1) = \tau_\Phi(f_2) \cdot \tau_\Phi(f_1).
	\end{align}
	holds in $\G$ across each pair $f_1:x \to y$ and $f_2:y \to z$ of composable morphisms in $\C$.
	\end{rem}

	\section{Group Actions on LP-Categories}\label{sec:grpact}
	
	The {\em automorphism group} $\text{Aut}(\C)$ of an LP-category $\C$ consists of all invertible LP-functors $F:\C \to \C$. The group structure on $\text{Aut}(\C)$ is obtained by the usual composition of such functors.
	
	\begin{defn}\label{def:grpact}
		An {\bf action} of a group $\G$ on an LP-category $\C$ is a group homomorphism $\alpha:\G \to \text{Aut}(\C)$; in particular, each group element $g$ is sent by $\alpha$ to an invertible functor $\alpha_g:\C \to \C$. Moreover, for every $g \in \G$ and non-identity morphism $f:x \to y$ in $\C$, we require:
		\begin{enumerate}
			\item $\alpha_g(x) \neq y$, and \label{item:grpact1}
			\item if $\alpha_g(x) = x$ then $\alpha_g(f) = f$. \label{item:grpact2}
		\end{enumerate}
	\end{defn}
	
	A large and interesting family of such actions is obtained whenever a group acts on a simplicial complex.
    
    \begin{example}\label{ex:grpactsimp}
	Fix a simplicial complex $X$ and a group $\G$ which acts on $X$ by invertible simplicial self-maps. Such an action is called {\em regular} (see \cite[Chapter III.1]{bredon1972compact}) if two conditions hold for every simplex $x$ --- writing $(v_0,\ldots,v_d)$ for the vertices of $x$,
    \begin{enumerate}
	\item every $g$ in the stabiliser $\G_x$ must fix the $v_i$ individually; and
	\item every simplex of the form $(g_0 v_0,\ldots,g_d v_d)$ for $g_i \in \G$ must lie in the orbit ${\G}x$.
    \end{enumerate}
    Any $\G$-action on $X$ can be made regular by passing to the second barycentric subdivision of $X$. Let $\fac{X}$ be the set of all simplices in $X$ partially ordered by the co-face relation $x \geq x'$. This poset is automatically a 1-category and hence an LP-category (with trivial partial orders on all its hom-sets). If $\G$ acts regularly on $X$, then the induced action on $\fac{X}$ satisfies both requirements of Definition \ref{def:grpact}.
	\end{example}
	
	Fix a group $\G$, an LP-category $\C$, and an action $\alpha:\G \to \text{Aut}(\C)$ as described above. In the interest of brevity, we will denote images of the form $\alpha_g(x)$ and $\alpha_g(f)$ by $gx$ and $gf$ respectively, thus omitting $\alpha$ from the notation altogether. For each object $x$ and morphism $f:x \to y$ in $\C$ we have the associated {\em object-orbits} and {\em morphism-orbits} given by
	\[
	{\G}x := \left\{g x \mid g \in \G\right\} \quad \text{and} \quad	{\G}f := \left\{g f \mid g \in \G\right\};
	\]
	similarly, we have the corresponding {\em stabilisers}
	\[
	\G_x := \left\{g \in \G \mid g x = x\right\} \quad \text{and} \quad \G_f := \left\{g \in \G \mid g f = f\right\},
	\] 
	both of which are subgroups of $\G$. It follows from Definition \ref{def:grpact}(2) that $\G_{f}$ is a subgroup of $\G_{x}$ for every morphism $f:x \to y$ in $\C$. 
	
	Given $g$ in $\G$ and $f:x \to y$ in $\C$, the source and target of $gf$ lie in ${\G}x$ and ${\G}y$ respectively. And for any pair of $\C$-morphisms $f:x \to y$ and $f':y \to z$, it is straightforward to check that the composite of any $f_1 \in {\G}f$ with any $f_2 \in {\G}f'$ must lie in the orbit ${\G}(f' \circ f)$. Thus, we have a well-defined (loopfree) 1-category $\Q$ whose objects are $\set{{\G}x \mid x \text{ is an object of }\C}$, and similarly for morphisms. In order to define a suitable quotient for the given $\G$-action on $\C$, it remains to upgrade $\Q$ to a genuine LP-category by specifying the desired partial orders on its morphism-sets.
	
	\begin{prop}\label{prop:quotposet}
		Consider the binary relation on the set of all morphisms ${\G}x \to {\G}y$ in $\Q$ prescribed by the following rule: ${\G}f \two {\G}f'$ whenever there exists some $g \in \G$ which satisfies $gf \two f'$ in $\C$. This binary relation forms a partial order on the set $\Q({\G}x,{\G}y)$.
	\end{prop}
	\begin{proof}
		We seek to establish reflexivity, anti-symmetry and transitivity. To this end, note first that ${\G}f \two {\G}f$ always holds in $\Q({\G}x,{\G}y)$ because $f = \text{id}_{\G} f \two f$. Next, if we have ${\G}f \two {\G}f' \two {\G}f$, then there exist some $g_0$ and $g_1$ in $\G$ satisfying $g_0f \two f'$ and $g_1f' \two f$, whence 
		\[
		g_1g_0 f \two g_1 f' \two f.
		\] Now, 
		$g_1g_0 f \two f$, so $g_1g_0$ fixes the source, hence $g_1g_0$ fixes $f$, by requirement (2) of Definition \ref{def:grpact}. So in fact $f' = g_1^{-1}f$ lies in the orbit ${\G}f$ and we have ${\G}f = {\G}f'$. Finally, assume that ${\G}f \two {\G}f'$ and ${\G}f' \two {\G}f''$ both hold in $\Q({\G}x,{\G}y)$, so $gf \two f'$ and $g'f' \two f''$ hold for some $g$ and $g'$ in $\G$. Then, we immediately have $g'g f \two f''$, and hence ${\G}f \two {\G}f''$ as desired.
	\end{proof}

	We have now arrived at the desired quotient category.
	
	\begin{defn}
		The {\bf quotient} LP-category $\Q = \C/\G$ corresponding to the action of $\G$ on $\C$ is defined by the following data:
		\begin{enumerate}
			\item its objects are $\G$-orbits of $\C$'s objects,
			\item its morphisms ${\G}x \to {\G}x'$ are $\G$-orbits of $\C$'s morphisms $x \to x'$, partially ordered by ${\G}f \two {\G}f'$ iff $gf \two f'$ holds in $\C$ for some $g$ in $\G$. 
		\end{enumerate} 
		The composite of ${\G}f_1:{\G}x_0 \to {\G}x_1$ and ${\G}f_2:{\G}x_1 \to {\G}x_2$ equals ${\G}(f_2 \circ f_1)$. 
	\end{defn}

	There is a surjective {\bf orbit functor} $\p_{\G}:\C \surj \Q$ (of LP-categories) that sends each $x$ to ${\G}x$ and each $f:x \to y$ to ${\G}f:{\G}x \to {\G}y$. The following notion arises when trying to find sections of $\p_{\G}$.
	
	\begin{defn}\label{def:liftfunc}
		A {\bf lift function} for the $\G$-action on $\C$ is an assignment $\lambda:\Q_\text{ob} \to \C_\text{ob}$ of a $\C$-object $\lambda(y)$ to each $\Q$-object $y$ for which $\p_{\G} \circ \lambda(y) = y$ holds.
	\end{defn}
	
	The reader is warned that (in contrast to the orbit functor $\p_{\G}$) a lift $\lambda$ need not extend to a functor in general. The difficulty here is that morphisms $y \to y'$ in $\Q$ may not lift to morphisms $\lambda(y) \to \lambda(y')$ in $\C$. For each $f:y \to y'$ in $\Q$, we are only guaranteed the existence of a (not necessarily unique) group element $g \in \G$ and a unique morphism $\lambda(f):\lambda(y) \to g\lambda(y')$ in $\C$ so that $\p_{\G}(\lambda(f)) = f$.
	
	\begin{defn}\label{def:transfer}
		A choice of {\bf transfer elements} for a given lift function $\lambda:\Q_{\text{ob}} \to \C_\text{ob}$ assigns to each morphism $f:y \to y'$ in $\Q$ an element $\sigma := \sigma(f)$ in $\G$ so that there is a unique morphism $\lambda(f):\lambda(y) \to \sigma^{-1}\lambda(y')$ in $\C$ satisfying $\mathbf{P}_{\G}(\lambda(f)) = f$. We implicitly require the transfer element $\sigma(\text{id}_y)$ assigned to each identity morphism in $\Q$ to equal the identity element $1_{\G}$ in $\G$.
	\end{defn}
	
    Given a choice of lift function and transfer elements for the action of a group $\G$ on an LP-category $\C$, there is a standard way to construct a pair $(\F{},\Phi)$ where $\F{}$ is a complex of groups on the quotient $\Q = \C/\G$ and $\Phi:\F{} \inj \ccog{\Q}{\G}$ is an injective morphism on the constant $\G$-valued complex of groups on $\Q$.
    
    \begin{defn}\label{def:actcog}
        Let $\G$ be a group acting on an LP-category $\C$ with quotient $\Q$; let $\lambda:\Q_\text{ob} \to \C_\text{ob}$ be an associated lift function and $\set{\sigma(h) \mid h:y \to y' \text{ in }\Q}$ a choice of transfer elements for $\lambda$. The {\bf complex of groups $\F{}$ associated to these choices} is obtained by assigning
        \begin{enumerate}
        \item to each object $y$ the stabilizer $\F{y} := \G_{\lambda(y)}$ of its lift;
        \item to each morphism $h:y \to y'$ the group homomorphism $\F{h}:\F{y} \to \F{y'}$ given by conjugation with $\sigma(h)$, i.e.~$\F{h}(k)=\sigma(h)k\sigma(h)^{-1}$ for all $k\in\F{y}$; 
        \item to each order relation $h\two h^\prime$ the group element $\sigma(h\two h^\prime)\coloneqq\sigma(h)\cdot \sigma(h^\prime)^{-1}$; and finally, 
        \item to each pair $h_1: y_1 \to y_2$ and $h_2: y_2 \to y_3$ of composable morphisms, the twisting element 
        \[
        \gamma_{\F{}}(h_1,h_2) := g(h_2) \cdot g(h_1) \cdot g(h_2 \circ h_1)^{-1}.
        \]
        \end{enumerate}
        Similarly, the {\bf associated monomorphism } $\Phi: \F{} \inj \ccog{\Q}{\G}$ is given by assigning
        \begin{enumerate}
            \item to each $y$ the inclusion $\G_{\lambda(y)} \inj \G$, and
            \item to each $h: y \to y'$ the twisting element $\tau_\Phi(h) := \sigma(h)$.
        \end{enumerate}
    \end{defn}

	It has been shown in \cite[Chapter III.C.2.9]{bridson2011metric} that $\F{}$ and $\Phi$ as described above satisfy the requirements of Definitions \ref{def:cog} and \ref{def:cogmap} respectively when defined over small loopfree categories; moreover, making a different choice of lift function and transfer elements produces an isomorphic pair $(\F,\Phi)$. Similar arguments confirm that the same property holds in our context. Moreover, we have the additional data of a partial order on each hom-set, so we check that $\F{}$ is well-defined.
	
	\begin{prop}\label{prop:canonicalcogwd}
	The assignments above yield a well-defined complex of groups and associated morphism over an LP category.
	\end{prop}
	
	The proof follows from Lemmas~\ref{lem:cancogfunctor}, \ref{lem:cancogtwisting} and \ref{lem:canmorph}.
	
	\begin{lemma}\label{lem:cancogfunctor}
	The assignment $\F{}_{yy^\prime}$ is a functor ${\Q(y,y')\to\grp(\F{y},\F(y'))}$ for every pair of objects $(y,y^\prime)$.
	\end{lemma}
	\begin{proof}
     For every $f:y\to y^\prime$, the map $\F{f}$ is an injective group homomorphism because it is conjugation by the chosen transfer element. The homomorphism $\F{f}$ maps from $\F{y}$ to $\F{y^\prime}$: take any $k$ in $\F{y}$ and recall that the transfer $\sigma(f)$ satisfies that there is a unique morphism $$\lambda(f):\lambda(y)\to \sigma(f)^{-1}\lambda(y^\prime)$$ such that $\p(\lambda(f))=f$. By Definition~\ref{def:grpact}\eqref{item:grpact2}, the morphism $$k\lambda(f):k\lambda(y)\to k\sigma(f)^{-1}\lambda(y^\prime)$$ is equal to $\lambda(f)$, since $k\lambda(y)=\lambda(y)$. Hence, $$k\sigma(f)^{-1}\lambda(y^\prime)=\sigma(f)^{-1}\lambda(y^\prime),$$ which implies that $\sigma(f)\cdot k\cdot \sigma(f)^{-1}\in\F{y^\prime}$, i.e., $\F{f}(k)\in\F{y^\prime}$.
	
	To show that the elements $\sigma\coloneqq \sigma(f\two f^\prime)$ are well-defined, we have to check that for any $k\in\F{y}$, we have $\F{f}(k)=\sigma\cdot\F{f^\prime}(k)\cdot\sigma^{-1}$. We have, 
	\begin{align*}
	    \sigma\cdot\F{f^\prime}(k)\cdot\sigma^{-1} &= \sigma(f)\sigma(f^\prime)^{-1}\cdot\left(\sigma(f^\prime)\cdot k\cdot \sigma(f^\prime)^{-1}\right)\cdot\sigma(f^\prime)\sigma(f)^{-1}\\
	    &= \sigma(f)\cdot k\cdot \sigma(f)^{-1}.
	\end{align*}
	
	Lastly, we check that $\F{}_{yy^\prime}$ satisfies the identity and associativity laws. The identity law follows immediately from the fact that the transfer element $\sigma(\text{id}_y)$ is the identity element in $\G$, so $\F{}_{yy^\prime}(\text{id}_y)=\text{id}_{\F{y^\prime}}$.
	
	For the associativity law to hold, we need to check that for any three morphisms $f,f^\prime,f^{\prime\prime}\in\Q(y,y^\prime)$ satisfying $f\two f^\prime\two f^{\prime\prime}$, the identity ${\sigma(f\two f^\prime)\cdot \sigma(f^\prime\two f^{\prime\prime})}={\sigma(f\two f^{\prime\prime})}$ holds: 
	\begin{align*}
	    \sigma(f\two f^\prime)\cdot\sigma(f^\prime\two f^{\prime\prime}) &= \left(\sigma(f)\sigma(f^\prime)^{-1} \right)\left(\sigma(f^\prime)\sigma(f^{\prime\prime})^{-1} \right)\\
	    &= \sigma(f)\sigma(f^{\prime\prime})^{-1}\\
	    &= \sigma(f\two f^{\prime\prime}),
	\end{align*}
	as desired.
	\end{proof}
	
	\begin{lemma}\label{lem:cancogtwisting}
	For any pair $(f_0,f_1)$ of composable morphisms, the twisting element $\gamma(f_0,f_1)$ is well-defined and satisfies the cocycle condition.
	\end{lemma}
	\begin{proof}
	For each pair of composable morphisms $ y\xrightarrow{f_0}y'\xrightarrow{f_1}y'',$ it is a routine calculation to check that the twisting element $\gamma(f_0,f_1)$ satisfies $$ \F{f_1}\circ\F{f_0}(k)=\gamma(f_0,f_1)\cdot\F{}(f_1\circ f_0)(k)\cdot\gamma(f_0,f_1)^{-1}$$ for any $k\in\F{y}$.
	
	The twisting elements also satisfy the cocycle condition: for any sequence of composable morphisms $$ y\xrightarrow{f_0}y'\xrightarrow{f_1}y''\xrightarrow{f_2}y'''$$ both $\gamma(f_1,f_2)\gamma(f_0,f_2\circ f_1)$ and ${\F{f_2}(\gamma(f_0,f_1))\gamma(f_1\circ f_0,f_2)}$ are equal to $$\sigma(f_2)\sigma(f_1)\sigma(f_0)\sigma(f_2\circ f_1\circ f_0)^{-1},$$ as desired.
	\end{proof}
	
	\begin{lemma}\label{lem:canmorph}
	The morphism $\Phi:\F{}\to\cst_{\mathcal{Q}}$ is well-defined and satisfies the coherence axiom.
	\end{lemma}
	\begin{proof}
	For any $f:y\to y'$, it follows immediately from the fact that $\Phi_y$ and $\Phi_{y'}$ are inclusions into $\G$ and $\F{f}$ is conjugation by $\gamma_\Phi=\sigma(f)$, that $\gamma_\Phi:\Phi_{y'}\circ\F{f}\two\Phi_y$.
	
	The identity axiom is immediately satisfied because the transfer element associated to any identity morphism is the identity element in $\G$.
	
	Lastly, we check the coherence axiom, for any sequence of composable morphisms $y\xrightarrow{f_0}y'\xrightarrow{f_1}y''$. The left-hand side of the coherence axiom (Equation~\eqref{eq:coherence}) yields:
	\begin{align*}
	    \Phi_{y''}\left(\gamma_{\F{}}(f_0,f_1)\right)\cdot \gamma_\Phi(f_1\circ f_0) &= \left(\sigma(f_1)\sigma(f_0)\sigma(f_1\circ f_0)^{-1} \right)\cdot \sigma(f_1\circ f_0) \\
	    &= \sigma(f_1)\sigma(f_0),
	\end{align*}
	and the right-hand side yields:
	\begin{align*}
	    \gamma_\Phi(f_1)\cdot\cst f_1\left(\gamma_\Phi(f_0) \right)\cdot\gamma_{\cst}(f_1,f_2) &= \sigma(f_1)\cdot \text{id}(\sigma(f_0))\cdot 1_{\G} \\
	    &= \sigma(f_1)\sigma(f_0),
	\end{align*}
	as required.
	\end{proof}
	
	\section{Development and the Basic Construction}\label{sec:dev}
	
	Our goal here is to prove the following result, which forms a mild enhancement of \cite[Chapter III.C, Theorem 2.13]{bridson2011metric} to the poset-enriched setting. Throughout, we fix an LP-category $\Q$ and a group $\G$. 
	
	\begin{prop}\label{prop:bascon}
		Let $\F{}:\Q \to \grp$ be a complex of groups and $\Phi:\F{} \to \ccog{\Q}{\G}$ a monomorphism to the constant $\G$-valued complex of groups over $\Q$. Associated to this pair $(\F{},\Phi)$ is a new LP-category $\D = \D(\F{},\Phi)$ equipped with a natural $\G$-action so that 
		\begin{enumerate}
		    \item the quotient $\D/\G$ is isomorphic to $\Q$; moreover, 
		    \item there exist choices of lift function and transfer elements for the $\G$-action on $\D$ which produce the pair $(\F{},\Phi)$; and finally, 
		  \item if $\F{}$ and $\Phi$ are the complex of groups and monomorphism associated to a $\G$-action on an LP-category $\C$ (subject to some choices), then there exists a $\G$-equivariant isomorphism of LP-categories between $\C$ and $\D$.
		\end{enumerate}
	\end{prop}
	
	The LP-category $\D$ whose existence and uniqueness (up to equivariant isomorphism) are guaranteed by the above result is called the {\bf development} of the given pair $(\F{},\Phi)$. For each object $x$ of $\Q$ and each group element $g$ in $\G$, let $\cset_x:\G \twoheadrightarrow \G/\Phi_x(\F{x})$ be the map (of sets) which assigns to each group element $g$ the corresponding left coset $g\Phi_x(\F{x})$. With this shorthand in place, the development $\D$ can be provisionally defined as follows:
	\begin{enumerate}
	    \item its objects are pairs $(x,\cset_x(g))$ where $x$ is an object of $\Q$ and $g$ is an element of $\G$;
	    \item the morphisms $(x,\cset_x(g)) \to (y,\cset_y(h))$ are all pairs $(f,\cset_x(g))$ where $f:x \to y$ is a morphism in $\Q$ and the following cosets coincide:
	    \[
	    \cset_y(h) = \cset_y\left(g\cdot \tau_\Phi(f)^{-1}\right).
	    \]
	    \end{enumerate}
	    Two such morphisms satisfy the partial order relation $(f,\cset_x(g)) \two (f',\cset_x(g))$ within $\D$ whenever $f \two f'$ holds in $\Q$, and $\tau_\Phi(f)\cdot\tau_\Phi(f^\prime)^{-1}$ lies in $\Phi_y(\F y)$. The composite 
	   	\begin{align}\label{eq:devcomp}
		\xymatrixcolsep{1in}
		\xymatrix{
			(x,\cset_{x}(g_1)) \ar@{->}[r]^-{(f_1,\cset_x(g_1))} & (y,\cset_y(g_2)) \ar@{->}[r]^-{(f_2,\cset_y(g_2))} & (z,\cset_z(g_3))
		}
		\end{align} 
		is given by $(f_2\circ f_1, \cset_x(g_1))$, and the group $\G$ acts on $\D$ as follows: the element $h \in \G$ sends each object $(x,\cset_x(g))$ to $(x,\cset_x(h \cdot g))$, and similarly for morphisms.

    We first check that the morphisms $(f,\cset_x(g))$ of $\D$ are well-defined, i.e., independent of the choice of representative element $g$ of the coset.
    
    \begin{prop}\label{prop:Dwelldef}
        Let $f:x \to y$ be a morphism in $\C$. If $\cset_x(g) = \cset_x(h)$ holds for $g,h$ in $\G$, then we also have
        \[
        \cset_y\left(g \cdot \tau_\Phi(f)^{-1}\right) = \cset_y\left(h \cdot \tau_\Phi(f)^{-1}\right).
        \]
    \end{prop}
    \begin{proof}
        Let $k$ be any element of the group $\F{x}$; since $\tau_\Phi(f)$ is a twisting element of the form $\F{f} \circ \Phi_y \two \Phi_x$ by Definition \ref{def:cogmap}, we have 
        \[
        \tau_\Phi(f)\cdot\Phi_x(k)\cdot\tau_\Phi(f)^{-1} = \Phi_y \circ \F{f}(k).
        \]
    Since $\F{f}$ takes values in $\F{y}$, the cosets $\cset_y(\tau_\Phi(f)^{-1})$ and $\cset_y(\Phi_x(k)\cdot \tau_\Phi(f)^{-1})$ are equal regardless of $k$, so the desired conclusion follows. 
    \end{proof}	
    
    Next, we show that the composition of $(f_1,\cset_x(g_1))$ with $(f_2,\cset_y(g_2))$, as described in \eqref{eq:devcomp}, produces a morphism $(x,\cset_x(g_1)) \to (z,\cset_z(g_3))$ in $\D$.
    
    \begin{prop}
    If $\cset_y(g_2) = \cset_y(g_1 \cdot \tau_\Phi(f_1)^{-1})$ and $\cset_z(g_3) = \cset_z(g_2 \cdot \tau_\Phi(f_2)^{-1})$ hold, then we also have $\cset_z(g_3) = \cset_z(g_1 \cdot \tau_\Phi(f_2 \circ f_1)^{-1})$.
    \end{prop}
    \begin{proof}
    Since the cosets $\cset_y(g_2)$ and $\cset_z(g_1 \cdot \tau_\Phi(f_1)^{-1})$ are equal, there exists some $h \in \F{z}$ satisfying $\tau_\Phi(f_1) \cdot g_1^{-1} \cdot g_2 = \Phi_y(h)$, whence
    \[
    g_2 = g_1 \cdot \tau_\Phi(f_1)^{-1} \cdot \Phi_y(h).
    \]
    Using this expression for $g_2$ along with our assumption that $\cset_z(g_3) = \cset_z(g_2 \cdot \tau_\Phi(f_2)^{-1})$, we obtain
    \begin{align*}
        \cset_z(g_3) &= \cset_z(g_1 \cdot \tau_\Phi(f_1)^{-1} \cdot \Phi_y(h) \cdot \tau_\Phi(f_2)^{-1}) & \\
        &= \cset_z(g_1 \cdot \tau_\Phi(f_1)^{-1} \cdot \tau_\Phi(f_2)^{-1}) & \text{ by Prop \ref{prop:Dwelldef}}, \\
        &= \cset_z(g_1 \cdot \tau_\Phi(f_2 \circ f_1)^{-1}) & \text{ by \eqref{eq:injid}},
    \end{align*}
    as desired.
    \end{proof}

    It is not difficult to confirm that $\Q$ is the quotient $\D/\G$. To obtain $(\F{},\Phi)$ for the $\G$-action on $\D$ as per Definition \ref{def:actcog}, one uses the lift function $x \mapsto (x,\cset_y(1_{\G}))$ and the transfer elements $f \mapsto \tau_{\Phi}(f)$ for each object $x$ and morphism $f$ of $\Q$. Any group element $h$ acts on $(x,\cset_y(g))$ by mapping it to $(x,\cset_y(h\cdot g))$. Thus, it only remains to establish the universal property from assertion (3) of Proposition \ref{prop:bascon}.
    
    \begin{prop}
    Let $\C$ be any LP-category upon which $\G$ acts with quotient $\C/\G$ equal to $\Q$. If there exists a choice of lifts and transfers for this action which produces $(\F,\Phi)$ as the associated complex of groups and monomorphism to $\ccog{\C}{\G}$, then there exists a $\G$-equivariant isomorphism of LP-categories $\D(\F{},\Phi) \cong_{\G} {\C}$.
    \end{prop}
    \begin{proof}
    Let $\lambda:\Q_\text{ob}\to\C_\text{ob}$ and $\set{\sigma(f) \mid f:x \to y \text{ in }\Q}$ be any choice of lift function and associated transfer elements for the $\G$-action on $\C$ that produce $\F{}$ and $\Phi$. We define a functor $\h: \D(\F{},\Phi)\to \C$ as follows: it sends each object $(x,\cset_x(g))$ in $\D$ to $g\lambda(x)$ in $\C$. The surjectivity of $\h$ on objects is clear; injectivity follows easily from the fact that $g\lambda(x)=h\lambda(x)$ if and only if $h^{-1}g$ is in $\Phi_x(\F{x})$, in which case $c_x(g)=c_x(h)$. 
    
    Given any morphism $f:x\to y$ in $\Q$, the transfer element $\tau_\Phi(f)$ satisfies that $$\lambda(f):\lambda(x)\to \tau_\Phi(f)^{-1}\lambda(y)$$ is the unique morphism in $\C$ with source $\lambda(x)$ such that $\p(\lambda(f))=f$. We set the $\h$-image of the morphism $(f,\cset_x(g))$ to be $g\lambda(f)$. 
    
    We check that $\h$ is order-preserving, i.e., that $$(f,c_x(g))\two (f^\prime,c_x(g))\ \text{if and only if}\ g\lambda(f)\two g\lambda(f^\prime).$$ If $(f,\cset_x(g))\two (f^\prime,\cset_x(g))$ in $\D$, then $f\two f^\prime$ in $\Q$ and $\tau_\Phi(f)\cdot\tau_\Phi(f^\prime)^{-1}\in\Phi_y(\F{y})$. By definition of the quotient, $f\two f^\prime$ in $\Q$ if and only if there exists $k\in\G$ such that $k\lambda(f)\two \lambda(f^\prime)$. But $\lambda(f)$ and $\lambda(f^\prime)$ have the same source and target, and there is a unique morphism in the orbit of $\lambda(f)$ with this source. Hence, $k=\text{id}_{\G}$ and $g\lambda(f)\two g\lambda(f^\prime)$. The other direction follows from the fact that, if $g\lambda(f)\two g\lambda(f^\prime)$ in $\C$, then they must have the same target and so $\tau_\Phi(f)\cdot\tau_\Phi(f^\prime)^{-1}$ is in $\Phi_y(\F{y})$.
    
     The fact that $\h$ is bijective on hom-sets follows from an argument similar to the one above, and some simple computations. The fact that $\h$ is $\G$-equivariant follows almost directly from the definition.
    \end{proof}
    
    We introduce a technical lemma relating the development of complexes of groups over isomorphic LP categories, that will be used in Section~\ref{sec:equhomequiv}.
    
    \begin{lemma}\label{lem:cogcommute}
    Suppose that $\E:X\to X'$ is an isomorphism of LP categories, and that ${\F{}:X\to\grp}$ and ${\F{}':X'\to\grp}$ are complexes of group such that the following diagram commutes:
     \[
    \begin{tikzcd}[row sep=huge, column sep = large]
    X \arrow[r, "\mathbf{E}"] \arrow[d, "\mathbf{F}"'] & X^\prime \arrow[ld, "\mathbf{F}^\prime"] \\
    \grp                                          &                                         
    \end{tikzcd}
    \]
    If $\Phi:\F{}\to\cst_{X}$ and $\Phi':\F'{}\to\cst_{X'}$ are morphisms of complexes of groups such that $\Phi_f = \Phi'_{\E f}$ for every morphism $f$ in $X$, then there is a $\G$-equivariant isomorphism of LP categories $${\D(\F{},\Phi)\cong_{\G}\D(\F'{},\Phi')}.$$
    \end{lemma}
    \begin{proof}
    First note that by the commutativity of the diagram, the groups $\F{x}$ and $\F'{\E x}$ are equal for every $x\in X$; hence, the cosets $g\cdot\Phi(\F{x})$ and $g\cdot\Phi'_{\E x}(\F'{\E x})$ are equal for any group element $g$, and so the indices $[\G:\Phi_x(\F{x})]$ and $[\G:\Phi'_{\E x}(\F'{\E x})]$ are equal.
    
    Each object $x\in X$ generates the set $\{(x,g\cdot\Phi_x(\F{x}))\ |\ \forall g\in\G \}$ of objects in $\D(\F{},\Phi)$. This is clearly in bijection with the set of objects ${\{(\E x,g\cdot\Phi'_{\E x}(\F'{\E x}))\ |\ \forall g\in\G \}}$ in $\D(\F'{},\Phi')$, since $\E$ is an isomorphism. 
    
    The poset of morphisms $$(x, g\cdot\Phi_x(\F{x}))\to (y, g\cdot\Phi_f^{-1}\Phi_y(\F{y}))$$ in $\D(\F{},\Phi)$ is isomorphic to the poset of morphisms $$(\E x, g\cdot\Phi'_{\E x}(\F'{\E x}))\to (\E y, g\cdot(\Phi'_{\E f})^{-1}\Phi'_{\E y}(\F'{\E y}))$$ in $\D(\F'{},\Phi')$ by the arguments above, in addition to the fact that we assume that $\Phi_f=\Phi'_{\E f}$.
    
    The correspondence between the partial order on morphisms in $\D(\F{},\Phi)$ and $\D(\F'{},\Phi')$ follows from the fact that $f\two f'$ in $X$ if and only if $\E{f}\two\E{f'}$ in $X'$, and $\Phi_f=\Phi'_{\E f}$ for any morphism $f\in X$.
    \end{proof}

	\section{Discrete Morse Theory and the Flow Category}\label{sec:flow}
	
	Let $X$ be a simplicial complex; we will write $x > x'$ to indicate that the simplex $x'$ is a face of the simplex $x$.
	
		\begin{defn}\label{def:acycmatch}
		A set $\Sigma = \set{(x_\bullet > x'_\bullet)}$ of simplex-pairs in $X$ is called an {\bf acyclic partial matching} if it satisfies three axioms:
		\begin{enumerate}
			\item {\bf dimension:} if $(x > x') \in \Sigma$, then $\dim x - \dim x' = 1$;
			\item {\bf partition:} if $(x > x') \in \Sigma$, then neither $x$ nor $x'$ lies in any other pair of $\Sigma$; and,
			\item {\bf acyclicity:} the transitive closure of the binary relation
			\[
			(x_0 > x_0') \btrt (x_1 > x_1') \text{ whenever } x_0 > x_1' \text{ in } X
			\]
			generates a partial order on $\Sigma$.
		\end{enumerate}
	\end{defn}
	
	Acyclic partial matchings are analogous to gradient-like vector fields on simplicial complexes; as such, they play a central role in Forman's {\em discrete Morse theory} \cite{forman1998morse, chari00discrete}. The simplices of $X$ which do not appear in any $\Sigma$-pair are called {\bf critical}, and the gradient paths from a critical simplex $w$ to a critical simplex $z$ are furnished by zigzags of $\btrt$-descending $\Sigma$ elements, e.g.,
	\[
	\zeta := (w > x'_0 < x_0 > x'_1 < x_1 > \cdots > x_k' < x_k > z). 
	\]
	Here each backward-pointing $x'_i < x_i$ corresponds to an element $(x_i > x'_i)$ in $\Sigma$, and the fact that $x'_{i+1}$ is a face of $x_i$ ensures $(x_i > x'_i) \btrt (x_{i+1} > x'_{i+1})$. An inductive argument establishes that $X$ is homotopy equivalent to a CW complex whose cells correspond bijectively with the critical simplices of $\Sigma$ --- see for instance \cite[Corollary 3.5]{forman1998morse}. To obtain a more explicit description of this homotopy equivalence in terms of $\Sigma$-zigzags (such as $\zeta$), we adopt the perspective of \cite{NANDA2019459}. The basic idea is to construct an LP-category whose objects are critical simplices, whose morphisms $w \to z$ consist of $\Sigma$-zigzags, and whose classifying space is homotopy equivalent to $X$.
		
	 Consider two simplices $x, x'$ in $X$. The set of {\em entrance paths} of $X$ from $x$ to $x'$, denoted $\ent{X}(x,x')$, consists of all strictly descending sequences of faces from $x$ to $x'$. Thus, each $f \in \ent{X}(x,x')$ has the form
	\[
	f = (x > x_1 > \cdots > x_k > x')
	\]
	for some length $k \geq 0$ where every $x_i$ in sight is a simplex of $X$. These entrance paths are partially ordered by refinement, so we have $f' \two f$ whenever $f'$ is obtained by removing some of the intermediate $x_i$ from $f$. In particular, we note that the only entrance path from each simplex $x$ to itself is the trivial path $(x)$ of length zero. 
	
	\begin{defn} The {\bf entrance path category} of $X$ is the LP-category $\ent{X}$ prescribed by the following data:
		\begin{enumerate}
			\item its objects are the simplices of $X$, and
			\item the poset of morphisms $x \to x'$ is $\ent{X}(x,x')$.
		\end{enumerate}
		The composite of $f := (x > x_0 > \cdots > x_k > x')$ and $f' := (x' > x_0' > \cdots > x_\ell > x'')$ is the concatenated path
		\[
		f' \circ f := (x > x_0 > \ldots > x_k > x' > x_0' > \cdots > x_\ell > x''),
		\]
		which evidently lies in the poset $\ent{X}(x,x'')$.
		For each cell $x$ of $X$, the singleton $(x)$ serves as the identity morphism, which is the unique element in the poset $\ent{X}(x,x)$
	\end{defn}
	
	The classifying space of $\ent{X}$ is homotopy equivalent to $X$ (see \cite[Proposition 3.3]{NANDA2019459}), and every acyclic partial matching $\Sigma = \set{(x_\bullet > x'_\bullet)}$ on $X$ corresponds to a collection of minimal entrance paths $\set{f_\bullet:x_\bullet \to x'_\bullet}$.
	
	\begin{defn}\label{def:catloc}
		The {\bf localisation} of $\ent{X}$ about an acyclic partial matching $\Sigma$ is the LP-category $\loc{X}$ whose objects are also the cells of $X$, while the 1-morphisms are equivalence classes of finite (but arbitrarily long) {\bf $\Sigma$-zigzags} $\gamma:w \to z$ in $\ent{X}$; each such zigzag has the form
		\[
		\xymatrixcolsep{.45in}
		\xymatrix{
		w \ar@{->}[r]^{h_0} & x'_0  & x_0  \ar@{->}[l]_{f_0} \ar@{->}[r]^{h_1} & x'_1 & \cdots \ar@{->}[l]_{f_1} \ar@{->}[r]^{h_k} & x'_k & x_k \ar@{->}[l]_{f_k} \ar@{->}[r]^{h_{k+1}} & z.
		}
		\]
		Here each backward-pointing $f_i$ is required to either lie in $\Sigma$ or be an identity, whereas the forward-pointing $h_i$ are unconstrained entrance paths. (The equivalence relation between such zigzags is defined in Remark \ref{rem:eqrel} below.) The order relation $\gamma \two \gamma'$ holds if there exist zigzag representatives which fit into an $\ent{X}$-diagram of the form
		\[
		\xymatrixcolsep{.45in}
		\xymatrixrowsep{.35in}
		\xymatrix{
		w \ar@{=}[d] \ar@{->}[r]^{h_0} & x'_0 \ar@{->}[d]_{\phi_0} & x_0 \ar@{->}[d]_{\phi_0} \ar@{->}[l]_{f_0} \ar@{->}[r]^{h_1} & x'_1 \ar@{->}[d]_{\phi'_1} & \cdots \ar@{->}[l]_{f_1} \ar@{->}[r]^{h_k} & x'_k \ar@{->}[d]_{\phi'_k} & x_k \ar@{->}[d]_{\phi_k} \ar@{->}[l]_{f_k} \ar@{->}[r]^{h_{k+1}} & z \ar@{=}[d] \\
		w \ar@{=>}[ur] \ar@{->}[r]_{h'_0} & y'_0 & y_0  \ar@{=>}[ul] \ar@{=>}[ur] \ar@{->}[l]^{f'_0} \ar@{->}[r]_{h'_1} & y'_1 & \cdots \ar@{=>}[ur]  \ar@{=>}[ul] \ar@{->}[l]^{f'_1} \ar@{->}[r]_{h'_k} & y'_k & y_k w \ar@{=>}[ul] \ar@{=>}[ur] \ar@{->}[l]^{f'_k} \ar@{->}[r]_{h'_{k+1}} & z 
		}
		\]
		Here all the vertical entrance paths $\phi_i, \phi_i'$ are also constrained to lie in $\Sigma \cup \set{\text{identities}}$, and all the diagonal double-arrows are to be interpreted as order relations in $\ent{X}$.
	\end{defn}
	
	\begin{rem}\label{rem:eqrel}
	   The equivalence relation among $\Sigma$-zigzags in the preceding definition is generated by (the transitive closure of) two elementary relations:
	   \begin{enumerate}
	       \item two zigzags are {\em horizontally} equivalent if they only differ by (backward or forward pointing) identity morphisms; and
	       \item two zigzags are {\em vertically} equivalent if they form the top and bottom rows of an $\ent{X}$-diagram
	       \[
		\xymatrixcolsep{.45in}
		\xymatrixrowsep{.35in}
		\xymatrix{
		w \ar@{=}[d] \ar@{->}[r]^{h_0} & x'_0 \ar@{->}[d]_{\phi_0} & x_0 \ar@{->}[d]_{\phi_0} \ar@{->}[l]_{f_0} \ar@{->}[r]^{h_1} & x'_1 \ar@{->}[d]_{\phi'_1} & \cdots \ar@{->}[l]_{f_1} \ar@{->}[r]^{h_k} & x'_k \ar@{->}[d]_{\phi'_k} & x_k \ar@{->}[d]_{\phi_k} \ar@{->}[l]_{f_k} \ar@{->}[r]^{h_{k+1}} & z \ar@{=}[d] \\
		w \ar@{=}[ur] \ar@{->}[r]_{h'_0} & y'_0 & y_0  \ar@{=}[ul] \ar@{=}[ur] \ar@{->}[l]^{f'_0} \ar@{->}[r]_{h'_1} & y'_1 & \cdots \ar@{=}[ur]  \ar@{=}[ul] \ar@{->}[l]^{f'_1} \ar@{->}[r]_{h'_k} & y'_k & y_k w \ar@{=}[ul] \ar@{=}[ur] \ar@{->}[l]^{f'_k} \ar@{->}[r]_{h'_{k+1}} & z 
		}
		\]
		(Note that all order relations are identities, so every square in sight commutes).
	   \end{enumerate}
	\end{rem}
	
	There is a localisation LP-functor
	\[
	\LL_\Sigma: \ent{X} \to \loc{X},
	\] 
	which sends each object to itself and each entrance path to its own equivalence class of $\Sigma$-zigzags. The subcategory of $\loc{X}$ spanned by all the critical simplices of $\Sigma$ is called the {\bf flow category} of $\Sigma$ and denoted $\flo{X}$. The flow category is naturally embedded into $\loc{X}$ via an inclusion LP-functor
	\[
	\mathbf{J}_\Sigma:\flo{X} \inj \loc{X}.
	\]
	Here is the main result of \cite{NANDA2019459}.

	\begin{thm} \label{thm:flocat}
		If $\Sigma$ is any acyclic partial matching on a simplicial complex $X$, then both LP-functors
		\[
		\xymatrixcolsep{.45in}
		\xymatrix{
		\ent{X} \ar@{->}[r]^-{\LL_\Sigma} & \loc{X} & \flo{X} \ar@{_{(}->}[l]_-{\mathbf{J}_\Sigma}
		}
		\]
		induce homotopy equivalence of classifying spaces. 
	\end{thm}
\noindent As an immediate consequence, the classifying space $|\Delta\flo{X}|$ is homotopy equivalent to $X$.
	
	\section{The Morse Complex of Groups}\label{sec:morsecog}
	
	In this section, we construct what we call the \emph{Morse complex of groups}: a complex of groups over the flow category $\flo{Y}$ associated to a certain class of acyclic partial matchings defined on a simplicial complex $Y$. The first step in this direction is to make an acyclic partial matching $\Sigma$ compatible with a complex of groups over $Y$
	
	\begin{defn}\label{def:fcompatible}
	    An acyclic partial matching $\Sigma$ on a simplicial complex $Y$ is \textbf{compatible} with a complex of groups $\F:\fac{Y}\to\grp$ if it satisfies the following conditions:
	    \begin{enumerate}
	        \item for every morphism $f:y\to y'$ in $\Sigma$, its $\mathbf{F}$-image $\F{f}$ is an isomorphism, and 
	        \item whenever we have 
	            \[
	            \begin{tikzcd}[row sep = small, column sep = small]
                w \arrow[rr, "g"] \arrow[rrdd, "g'"'] &               & y \arrow[dd, "f"] &  &           &  & y \arrow[dd, "f"'] \arrow[rrdd, "h'"] &    &   \\
                                                            & {} \arrow[ru, equals, shorten=2mm] &                   &  & \text{or} &  &                                             & {} &   \\
                                                            &               & y'                 &  &           &  & y' \arrow[rr, "h"'] \arrow[ru, equals, shorten=1mm]               &    & z
                \end{tikzcd}
                \]
                for arbitrary morphisms $g,g',h, h'$, the $\mathbf{F}$-images are, respectively,
                \[ 
                \begin{tikzcd}[row sep = small, column sep = small]
                \F{w} \arrow[rr, "\F{g}"] \arrow[rrdd, "\F{g'}"'] &               & \F{y} \arrow[dd, "\F{f}"] &  &           &  & \F{y} \arrow[dd, "\F{f}"'] \arrow[rrdd, "\F{h'}"] &    &   \\
                                                            & {} \arrow[ru, equals, shorten=2mm] &                   &  & \text{and} &  &                                             & {} &   \\
                                                            &               & \F{y'}                 &  &           &  & \F{y'} \arrow[rr, "\F {h}"'] \arrow[ru, equals, shorten=1mm]               &    & \F{z}
                \end{tikzcd}
                \]
	    \end{enumerate}
	\end{defn}

	Throughout this section, $\F{}$ denotes the complex of groups associated to a regular $\G$-action on $X$ (for some choice of lifts and transfers), $\Phi:\F{}\to\cst_Y$ denotes the associated monomorphism, and $\Sigma$ is an $\F{}$-compatible acyclic partial matching on the quotient $Y$.

	\begin{defn}\label{def:mcog}
	 The {\bf Morse complex of groups} associated to $\Sigma$ is the pseudofunctor $\M{}:\flo{Y}\to\grp$ prescribed by the following data:
	\begin{enumerate}
    \item To each object $y\in\flo{Y}$, the group $\M{y} \coloneqq \F{y}$.
    \item For each equivalence class of 1-morphisms represented by some zigzag $\zeta$ \[ 
        \begin{tikzcd}
             w \arrow[r, "h_0"] & y_0 & x_0 \arrow[l, "f_0"'] \arrow[r, "h_1"] & y_1 & \cdots \arrow[l, "f_1"'] & x_k \arrow[l, "f_k"'] \arrow[r, "h_{k+1}"] & z
        \end{tikzcd}
        \]
        the homomorphism $\M{\zeta} : \M{w}\to\M{z}$ is the composite  $$\F{h_{k+1}}\circ\F{f_k}^{-1}\circ\cdots\circ\F{h_1}\circ\F{f_0}^{-1}\circ\F{h_0},$$
        i.e., $\M{\zeta}$ is conjugation by the product of transfer elements $$\sigma(\zeta)\coloneqq \sigma(h_{k+1})\cdot \sigma(f_k)^{-1}\cdots \sigma(h_1)\cdot \sigma(f_0)^{-1}\cdot \sigma(h_0).$$
        
        \item For any two morphisms $\zeta,\zeta^\prime:w\to z$ in $\flo{Y}$ satisfying $\zeta^\prime\Rightarrow\zeta$, fitting into a diagram
        \begin{equation}\label{diag:2morphisms}
        \begin{tikzcd}
            w \arrow[r, "h_0"] \arrow[d]    & y_0 \arrow[d, "u_0"] & x_0 \arrow[l, "f_0"'] \arrow[r, "h_1"] \arrow[d, "v_0"]         & y_1 \arrow[d, "u_1"] & \cdots \arrow[l, "f_1"']            & x_k \arrow[l, "f_k"'] \arrow[r, "h_{k+1}"] \arrow[d, "v_k"]         & z \arrow[d] \\
            w \arrow[r, "h'_0"'] \arrow[ru, Rightarrow] & y'_0                 & x'_0 \arrow[l, "f'_0"] \arrow[r, "h'_1"'] \arrow[ru, Rightarrow] \arrow[lu, Rightarrow] & y'_1                 & \cdots \arrow[l, "f'_1"] \arrow[lu, Rightarrow] & x'_k \arrow[l, "f'_k"] \arrow[r, "h'_{k+1}"'] \arrow[lu, Rightarrow] \arrow[ru, Rightarrow] & z          
        \end{tikzcd}
        \end{equation}
        where $\zeta,\zeta^\prime$ form the top and bottom rows respectively, ${\M{(\zeta^\prime\Rightarrow\zeta)}:\M{\zeta^\prime}\Rightarrow\M\zeta}$ is the product $\sigma(\zeta^\prime)\cdot\sigma(\zeta)^{-1}.$
        \item For any composable morphisms $y\xrightarrow{\zeta}y'\xrightarrow{\zeta^\prime}y''$, the twisting element is $$\gamma_{\M{}}(\zeta,\zeta')=\sigma(\zeta')\cdot\sigma(\zeta)\cdot\sigma(\zeta'\circ\zeta)^{-1}.$$
\end{enumerate}
\end{defn}
	 
	 Whilst the complex of groups in Definition~\ref{def:actcog} summarises a group action, the Morse version tracks how this layer of data changes after applying discrete Morse theory to the quotient. The associated monomorphism $\Psi:\M{}\to\cst_{\flo{Y}}$ assigns
\begin{enumerate}
    \item to each object $y\in\flo{Y}$ the inclusion $\M{y}\hookrightarrow\G$, and
    \item to each $\zeta:y\to y'$ the twisting element $\tau_{\Psi}(\zeta)\coloneqq \sigma(\zeta)$.
\end{enumerate}

\begin{prop}\label{prop:morsecogwd}
$\M{}$ is a well-defined complex of groups.
\end{prop}
\begin{proof}
This follows from a sequence of claims:
\begin{enumerate}
    \item \emph{Claim:} The 1-morphisms are well-defined, that is, for two zigzags $\zeta\sim\zeta^\prime$, we have $\M{\zeta} \equiv\M{\zeta^\prime}$, and every $\M{\zeta}$ is injective. \\
    \emph{Proof:} If $\zeta$ and $\zeta^\prime$ are horizontally equivalent then clearly $\M{\zeta}\equiv\M{\zeta^\prime}$ because $\F{\text{id}_y} = \text{id}_{\F{y}}$ for all objects $y\in \fac{Y}$. If $\zeta$ and $\zeta^\prime$ are vertically equivalent, then they fit into a diagram as follows.
    \[ 
    \begin{tikzcd}
        w \arrow[r, "g_0"] \arrow[d]    & y_0 \arrow[d, "u_0"] & x_0 \arrow[l, "f_0"'] \arrow[r, "g_1"] \arrow[d, "v_0"]         & y_1 \arrow[d, "u_1"] & \cdots \arrow[l, "f_1"']            & x_k \arrow[l, "f_k"'] \arrow[r, "g_{k+1}"] \arrow[d, "v_k"]         & z \arrow[d] \\
        w \arrow[r, "g'_0"'] \arrow[ru, equals] & y'_0                 & x'_0 \arrow[l, "f'_0"] \arrow[r, "g'_1"'] \arrow[ru, equals] \arrow[lu, equals] & y'_1                 & \cdots \arrow[l, "f'_1"] \arrow[lu, equals] & x'_k \arrow[l, "f'_k"] \arrow[r, "g'_{k+1}"'] \arrow[lu, equals] \arrow[ru, equals] & z          
    \end{tikzcd}
    \]
    By the second compatibility criterion for $\Sigma$ (Definition~\ref{def:fcompatible}), for each backward-pointing column we have $$\F{u_i}\circ\F{f_i} = \F{f_i^\prime\circ v_i} = \F{f_i^\prime}\circ\F{v_i},$$ and similarly, for each forward-pointing column we have $$\F{u_i}\circ\F{h_i}=\F{h_i^\prime}\circ\F{v_{i-1}},$$ with the convention that $v_{-1}=\text{id}_z$ and $u_{k+1}=\text{id}_w$. These equations yield the relations
    \begin{align*}
        \F{f_i}^{-1} &= \F{v_i}^{-1}\circ\F{f_i^\prime}^{-1}\circ\F{u_i} \\
        \F{h_i} &= \F{u_i}^{-1}\circ\F{h_i^\prime}\circ\F{v_{i-1}},
    \end{align*}
    for all suitable $i$, and are well-defined since $f_i,u_i,v_i$ are either identities or are in $\Sigma$,  so their $\F{}$-images are isomorphisms by the first compatibility criterion. Substituting into $\M{\zeta}$, we have that
    \begin{align*}
        \M{\zeta} &= \F{h_{k+1}}\circ\F{f_k}^{-1}\circ\cdots\circ\F{h_1}\circ\F{f_0}^{-1}\circ\F{h_0} \\
        &= \left(\F{h_{k+1}^\prime}\circ\F{v_k} \right)\circ\cdots\circ\left(\F{v_0}^{-1}\circ\F{f_0^\prime}^{-1}\circ\F{u_0} \right)\circ\left(\F{u_0}^{-1}\circ\F{h_0^\prime} \right) \\
        &= \F{h_{k+1}^\prime}\circ\F{f_k^\prime}^{-1}\circ\cdots\circ\F{h_1^\prime}\circ\F{f_0^\prime}^{-1}\circ\F{h_0^\prime} \\
        &= \M{\zeta^\prime},
    \end{align*}
    as required. By construction, each $\M\zeta$ is the composite of finitely many injective homomorphisms, and is therefore injective. \\
    
    \item \emph{Claim:} The 2-morphism ${\M{(\zeta^\prime\Rightarrow\zeta)}}$ satisfies $$\M{\zeta^\prime}(k) = \M{(\zeta^\prime\Rightarrow\zeta)}\cdot\M{\zeta}(k)\cdot\M{(\zeta^\prime\Rightarrow\zeta)}^{-1},$$ for all $k\in\M{w}$.   \\
    \emph{Proof:} This follows from a routine calculation. \\
    
    \item \emph{Claim:} $\M{}$ satisfies the cocycle condition. \\
    \emph{Proof:} If $\zeta_0, \zeta_1$ and $\zeta_2$ are three composable $\Sigma$-zigzags $\bullet\xrightarrow{\zeta_0}\bullet\xrightarrow{\zeta_1}\bullet\xrightarrow{\zeta_2}\bullet$, then both $\gamma(\zeta_1,\zeta_2)\cdot\gamma(\zeta_0,\zeta_2\circ\zeta_1)$ and $\M{\zeta_2}(\gamma(\zeta_0,\zeta_1))\cdot\gamma(\zeta_1\circ\zeta_0,\zeta_2)$ are equal to $$\sigma(\zeta_2)\sigma(\zeta_1)\sigma(\zeta_0)\sigma(\zeta_2\circ\zeta_1\circ\zeta_0)^{-1},$$
    as desired.
\end{enumerate}
This concludes the argument.
\end{proof}

It follows from similar computations that $\Psi$ is a well-defined monomorphism.

\begin{rem}
We can define the 2-morphisms ${\M{(\zeta^\prime\Rightarrow\zeta)}}$ in a more intuitive way: for any pair of morphisms fitting into a diagram like \eqref{diag:2morphisms}, the functor $\F{}$ maps it to a diagram of the form:
\begin{equation*}\label{diag:teardrop}
    \begin{tikzcd}[row sep=0.5cm, column sep=0.8cm]
                                                                                                  & {}            &            & {}            &                                                                                                                                                                                                               & {}                    &        & {}            &                                                                                                                                                                                                              & {}                        &   \\
    \F{w} \arrow[rr, "\F{u_0}\circ\F{h_0}", bend left=74] \arrow[rr, "\F{h_0^\prime}"', bend right=74] &               & \F{y_0^\prime} &               & \F{x_0} \arrow[ll, "\F{u_0}\circ \F{f_0}"', bend right=74] \arrow[rr, "\F{u_1}\circ\F{h_1}", bend left=74] \arrow[ll, "\F{f_0^\prime}\circ\F{v_0}", bend left=74] \arrow[rr, "\F{h_1^\prime}\circ\F{v_0}"', bend right=74] &                       & \cdots &               & \F{x_k} \arrow[ll, "\F{u_k}\circ\F{ f_k}"', bend right=74] \arrow[rr, "\F{h_{k+1}}", bend left=74] \arrow[ll, "\F{f_k^\prime}\circ\F{v_k}", bend left=74] \arrow[rr, "\F{h_{k+1}^\prime}\circ\F{v_k}"', bend right=74] &                           & \F{z} \\
                                                                                                  & {} \arrow[uu, "t_0"', Rightarrow] &            & {} \arrow[uu, "1_{\G}"', Rightarrow] &                                                                                                                                                                                                               & {} \arrow[uu, "t_1"', Rightarrow] &        & {} \arrow[uu, "1_{\G}"', Rightarrow] &                                                                                                          []                                                                                                    & {} \arrow[uu, "t_{k+1}"', Rightarrow] &  
    \end{tikzcd}
    \end{equation*}
    where each $t_i$ is the 2-morphism assigned by $\F{}$. We can define ${\M{(\zeta^\prime\Rightarrow\zeta)}}$ to be the horizontal composition $t_0\cdots t_{k+1}$ of the $t_i$'s across this diagram. This definition is, in fact, equivalent to the one given above, but we omit the details here.
\end{rem}

	\section{Equivariant Homotopy Equivalence}\label{sec:equhomequiv}
	
	We are now ready to state our main result, which combines the Basic Construction with discrete Morse theory:
	
	\begin{thm}\label{thm:homequiv}
	    Given a regular simplicial action $\G$ on $X$, let $\F{}$ and $\Phi$ denote the complex of groups and morphism associated to this action, for some choice of lift and transfers. If $\Sigma$ is an $\F{}$-compatible acyclic partial matching on the quotient $Y$, then the Morse complex of groups ${\M:\flo{Y}\to\grp}$ and associated monomorphism $\Psi:\M\to\cst_{\flo{Y}}$ satisfy the following property: the classifying space of the development $|\Delta\D(\M,\Psi)|$ is $\G$-equivariantly homotopy-equivalent to $X$.
	\end{thm}

Throughout this section, we assume that the hypotheses of Theorem~\ref{thm:homequiv} hold. The proof of this theorem follows from a sequence of results, and the approach is summarised in Figure~\ref{fig:thmpf}. The high-level strategy is to lift the acyclic partial matching $\Sigma$ on $Y$ to a matching $\tilde{\Sigma}$ in $X$, to push the $\G$-action on $X$ to $\mathbf{Flo}_{\tilde{\Sigma}}[X]$, and to show that $\mathbf{M}$ arises as the complex of groups for this action, for some choices dependent on the original choices made to construct $\F{}$.

\begin{figure}[h!]
\begin{tikzcd}[row sep = huge, column sep = huge]
X \arrow[r, "\mathbf{L}_{\tilde{\Sigma}}"] \arrow[d, "\mathbf{P}"'] & {\mathbf{Loc}_{\tilde{\Sigma}}[X]} \arrow[d, "\tilde{\mathbf{P}}"]
& {\mathbf{Flo}_{\tilde{\Sigma}}[X]} \arrow[l, "\mathbf{J}_{\tilde{\Sigma}}"'] \arrow[rd, "\mathbf{P}"] \arrow[d, "\tilde{\mathbf{P}}"] &                                                                  \\
Y \arrow[r, "\mathbf{L}_\Sigma"']                            & {\mathbf{Loc}_\Sigma[Y]}           & {\mathbf{Flo}_\Sigma[Y]} \arrow[l, "\mathbf{J}_\Sigma"]                                                                          & {\mathbf{Flo}_{\tilde{\Sigma}}[X]/G} \arrow[l, "\mathbf{E}"]
\end{tikzcd}
\caption{We prove Theorem~\ref{thm:homequiv} by first lifting the acyclic partial matching on $Y$ to a $\G$-equivariant one on $X$, and then by showing that the development of $(\mathbf{M},\Psi)$ is isomorphic to $\mathbf{Flo}_{\tilde{\Sigma}}[X]$.}
\label{fig:thmpf}
\end{figure}

\begin{lemma}\label{lem:sigmatilde}
If $\Sigma=\{f_i:x_i\to y_i\}$ is an acyclic partial matching on the quotient $Y$, then $\tilde{\Sigma}\coloneqq \bigcup_i \G f_i$ is an acyclic partial matching on $X$.
\end{lemma}
\begin{proof}
We have to prove that $\tilde{\Sigma}$ satisfies the dimension, partition and acyclicity conditions from Definition~\ref{def:acycmatch}. The dimension property follows immediately since $\text{dim}(\tilde{z})=\text{dim}(z)$ for any $\tilde{z}$ satisfying $\p(\tilde{z})=z$. 

Since $\Sigma$ satisfies the partition property, it follows that ${\G f_i\cap \G f_j=\emptyset}$ for $i\neq j$. So it remains to show that for a given $f\in\Sigma$, the orbit $\G f\subset\tilde{\Sigma}$ contains no repeated objects. This could occur in precisely two ways, as illustrated in Figure~\ref{fig:sigmatildetwo}. Without loss of generality, we can assume that one of the morphisms violating the partition property has source $\lambda(x)$.
\begin{figure*}[h!]
\centering
\begin{tikzcd}[column sep=tiny]
                        &  & (1)                                                                                             &  &                         &  &  &                                                   & (2)                               &                                                 \\
\sigma(f)^{-1}\lambda(y) &  & g\lambda(x) \arrow[rr, "g\lambda(f)"', bend right] \arrow[ll, "\lambda(f)", bend left] \arrow[d,equals] &  & g\sigma(f)^{-1}\lambda(y) &  &  & g\lambda(x) \arrow[rd, "g\lambda(f)"', bend right] &                                   & \lambda(x) \arrow[ld, "\lambda(f)", bend left] \\
                        &  & \lambda(x)                                                                                     &  &                         &  &  &                                                   & g\sigma(f)^{-1}\lambda(y) \arrow[d,equals] &                                                 \\
                        &  &                                                                                                 &  &                         &  &  &                                                   & \sigma(f)^{-1}\lambda(y)           & \\
                        & & & & & & & & &
\end{tikzcd}
\caption{The two possible ways in which $\tilde{\Sigma}$ could violate the partition property of an acyclic partial matching.}
\label{fig:sigmatildetwo}
\end{figure*}

Case 1 cannot occur because there is a unique morphism with a given source in $\lambda(f)$'s orbit, so $g\lambda(f)=\lambda(f)$. Suppose we are in Case 2; then $g\sigma(f)^{-1}\lambda(y)=\sigma(f)^{-1}\lambda(y)$ and so $\sigma(f)g\sigma(f)^{-1}\in\F{y}$. Recall that since $\Sigma$ is $\F{}$-compatible, the homomorphism $\F{f}$ is an isomorphism for every $f\in\Sigma$. Hence, there is a $k\in\F{x}$ such that $$\F{f}(k)=\sigma(f)\cdot k\cdot\sigma(f)^{-1}=\sigma(f)\cdot g\cdot\sigma(f)^{-1}.$$ So $k=g$, and $g\lambda(x)=\lambda(x)$ and so we satisfy the partition property.

Suppose that $\tilde{\Sigma}$ is not acyclic, i.e., that there exists a sequence (writing $\tilde{x}_i$ for $\lambda(x_i)$ and similarly for $y_i$):
$$g_1\sigma(f_1)^{-1}\tilde{y}_1<g_1\tilde{x}_1>g_2\sigma(f_2)^{-1}\tilde{y}_2<\cdots>g_n\sigma(f_n)^{-1}\tilde{y}_n<g_n\tilde{x}_n,$$ with $g_n\tilde{x}_n>g_1\sigma(f_1)^{-1}\tilde{y}_1.$ The projection of this sequence 
$$y_1<x_1>y_2<x_2>\cdots>y_n<x_n$$ to the quotient has the property that $x_n>y_1$. 

This projected sequence could contain repeats if any of the $\tilde{x}_i$ are in the same orbit, for example; if, say, the pair $y_i<x_i$ is equal to $y_j<x_j$ for some $i<j$, remove from the projected sequence the entire portion between these two pairs, inclusive of $y_j<x_j$.

If this sequence has length strictly bigger than one, it contradicts the acyclicity of $\Sigma$. This reduced sequence will have length equal to one if and only if all pairs involving $\tilde{y}_i$ and$\tilde{x}_i$ lie in the same orbit. But this contradicts condition~\eqref{item:grpact2} from Definition~\ref{def:grpact}, because this would require the existence of relations of the form $g_1\tilde{y}<\tilde{x}>g_2\tilde{y}$ for some group elements $g_1$ and $g_2$, i.e., it would require two distinct morphisms in the same orbit with the same source.
\end{proof}

Using the action on $X$, we can define an action on $\mathbf{Loc}_{\tilde{\Sigma}}[X]$: the objects of both categories are the same, and the action on objects remains the same. On morphisms, each $g\in\G$ acts individually on each arrow in the zigzag. More precisely, if $\zeta$ is the  $\tilde{\Sigma}$ zigzag
    \[ 
    \begin{tikzcd}
         w \arrow[r, "h_0"] & y_0 & x_0 \arrow[l, "f_0"'] \arrow[r, "h_1"] & y_1 & \cdots \arrow[l, "f_1"'] & x_k \arrow[l, "f_k"'] \arrow[r, "h_{k+1}"] & z,
    \end{tikzcd}
    \]
    then the morphism $g\cdot\zeta$ is defined to be
    \[ 
    \begin{tikzcd}
         gw \arrow[r, "g\cdot h_0"] & gy_0 & gx_0 \arrow[l, "g\cdot f_0"'] \arrow[r, "g\cdot h_1"] & gy_1 & \cdots \arrow[l, "g\cdot f_1"'] & gx_k \arrow[l, "g\cdot f_k"'] \arrow[r, "g\cdot h_{k+1}"] & gz.
    \end{tikzcd}
    \]
The following result confirms that this action is well-defined, and also that it restricts to $\mathbf{Flo}_{\tilde{\Sigma}}[X]$.

\begin{lemma}\label{lem:floaction}
The induced $\G$-action on $\mathbf{Flo}_{\tilde{\Sigma}}[X]$ is well-defined.
\end{lemma}
\begin{proof}
First note that the objects of $\mathbf{Loc}_{\tilde{\Sigma}}[X]$ are the same as for $X$, and so is the action on the objects; for the action to be well-defined on $\mathbf{Flo}_{\tilde{\Sigma}}[X]$, we need that the orbits of critical objects contain only critical objects, and similarly for non-critical objects. But this follows trivially from the fact that $\tilde{\Sigma}$ is $\G$-invariant by construction. 

The $\G$-invariance of $\tilde{\Sigma}$ also guarantees that $\tilde{\Sigma}$-zigzags are sent to $\tilde{\Sigma}$-zigzags under the action. That is, if $\zeta$ is the  $\tilde{\Sigma}$ zigzag
    \[ 
    \begin{tikzcd}
         w \arrow[r, "h_0"] & y_0 & x_0 \arrow[l, "f_0"'] \arrow[r, "h_1"] & y_1 & \cdots \arrow[l, "f_1"'] & x_k \arrow[l, "f_k"'] \arrow[r, "h_{k+1}"] & z,
    \end{tikzcd}
    \]
    then the morphism $g\cdot\zeta$
    \[ 
    \begin{tikzcd}
         gw \arrow[r, "g\cdot h_0"] & gy_0 & gx_0 \arrow[l, "g\cdot f_0"'] \arrow[r, "g\cdot h_1"] & gy_1 & \cdots \arrow[l, "g\cdot f_1"'] & gx_k \arrow[l, "g\cdot f_k"'] \arrow[r, "g\cdot h_{k+1}"] & gz,
    \end{tikzcd}
    \]
    is also a $\tilde{\Sigma}$-zigzag because each $g\cdot f_i$ is in $\tilde{\Sigma}$ by definition of $\tilde{\Sigma}$. 
    
    Next we check that if $\zeta\sim\xi$ then $g\zeta\sim g\xi$, for any $g\in\G$. If $\zeta$ and $\xi$ are horizontally equivalent, then their $g$-images are also horizontally equivalent because $g\cdot \text{id}_z=\text{id}_{g\cdot z}$ for any object $z$. If $\zeta$ and $\xi$ are vertically equivalent, then they fit into a diagram as follows
    \[
    \begin{tikzcd}
    w \arrow[d,equals] \arrow[r]  & y_0 \arrow[d] & x_0 \arrow[d] \arrow[r] \arrow[l]                    & \cdots \arrow[r] & z \arrow[d,equals]  \\
    w \arrow[r] \arrow[ru,equals] & y_0^\prime    & x_0^\prime \arrow[r] \arrow[l] \arrow[lu,equals] \arrow[ru,equals] & \cdots \arrow[r] & z \arrow[lu,equals]
    \end{tikzcd}
    \]
    which maps to the following under $g$:
    \[
    \begin{tikzcd}
    g\cdot w \arrow[d,equals] \arrow[r]  & g\cdot y_0 \arrow[d] & g\cdot x_0 \arrow[d] \arrow[r] \arrow[l]                    & \cdots \arrow[r] & g\cdot z \arrow[d,equals]  \\
    g\cdot w \arrow[r] \arrow[ru,equals] & g\cdot y_0^\prime    & g\cdot x_0^\prime \arrow[r] \arrow[l] \arrow[lu,equals] \arrow[ru,equals] & \cdots \arrow[r] & g\cdot z \arrow[lu,equals]
    \end{tikzcd}
    \]
    The equality of composites in each square follows from the fact that the original action on $X$ respects the partial order on morphisms, and each composite consists of morphisms from $X$. Hence, the morphism $g\cdot\zeta$ is vertically equivalent to $g\cdot\xi$ and so the induced action is well-defined on equivalence classes of morphisms.
    
    Lastly, we need to check that the induced action still satisfies the two properties in Definition~\ref{def:grpact}. Suppose that $\zeta$ is the morphism \[
    \begin{tikzcd}
        w \arrow[r, "h_0"] & y_0 & x_0 \arrow[l, "f_0"'] \arrow[r, "h_1"] & y_1 & \cdots \arrow[l, "f_1"'] & x_k \arrow[l, "f_k"'] \arrow[r, "h_{k+1}"] & z,
    \end{tikzcd}
    \]
    and suppose also that $g\cdot w=z$. Then the morphism $g\cdot\zeta$ 
    \[ 
    \begin{tikzcd}
         z \arrow[r, "g\cdot h_0"] & gy_0 & gx_0 \arrow[l, "g\cdot f_0"'] \arrow[r, "g\cdot h_1"] & gy_1 & \cdots \arrow[l, "g\cdot f_1"'] & gx_k \arrow[l, "g\cdot f_k"'] \arrow[r, "g\cdot h_{k+1}"] & gz
    \end{tikzcd}
    \]
    projects under the orbit functor to a cycle in $\flo{Y}$, contradicting the acyclicity of $\Sigma$. 
    
    If instead that there is some $g\in\G$ fixing $w$, then $gh_0:w\to gy_0$. But, since the group action on $X$ satisfies property~\eqref{item:grpact2} of Definition~\ref{def:grpact}, we must have that $g h_0=h_0$, and in particular that $gy_0=y_0$. Since $\tilde{\Sigma}$ satisfies the partition property, we have $gx_0=x_0$ and $gf_0=f_0$, because $y_0$ cannot appear as the target of any other morphism in $\tilde{\Sigma}$. Applying this argument to the rest of the chain of morphisms forming $\zeta$ implies that, if $g$ fixes $w$, then $g$ also fixes $\zeta$.
\end{proof}

We will construct a functor $\E$ between the quotient $\mathbf{Flo}_{\tilde{\Sigma}}[X] /\G$ and the flow category on the original quotient $\flo{Y}$, and show it is an isomorphism of categories. This will enable us to relate our Morse complex of groups and the canonical complex of groups over the orbit space of the action on $\mathbf{Flo}_{\tilde{\Sigma}}[X]$.

Let $\p$ denote the usual orbit functor. Using $[\bullet]_X$ to denote orbit classes, define $\mathbf{\tilde{P}}:\mathbf{Flo}_{\tilde{\Sigma}}[X]\to\flo{Y}$ to be the functor mapping each $\tilde{\Sigma}$-zigzag $$\tilde{h}_{k+1}\circ \tilde{f}_{k}^{-1}\circ\cdots\circ \tilde{f}_{0}^{-1}\circ \tilde{h}_0$$ to the $\Sigma$-zigzag $$[\tilde{h}_{k+1}]_X\circ [\tilde{f}_{k}^{-1}]_X\circ\cdots\circ [\tilde{f}_{0}^{-1}]_X\circ [\tilde{h}_0]_X.$$ 

\begin{lemma}\label{lemma:ptilde}
$\mathbf{\tilde{P}}$ is a surjective LP functor.
\end{lemma}
\begin{proof}
It is a routine calculation to check that $\tilde{\mathbf{P}}$ is an LP functor. Every constituent morphism ${h:x\to y}$ in a $\Sigma$-zigzag can be lifted to a morphism in its orbit class with prescribed source $g\lambda(x)$, for any $g\in\G$. Also, any morphism in $\Sigma$ is lifted to a morphism in $\tilde{\Sigma}$. Hence, we can find a $\tilde{\Sigma}$ morphism $\tilde{\zeta}$ in $\mathbf{Flo}_{\tilde{\Sigma}}[X]$ such that $\tilde{\mathbf{P}}(\tilde{\zeta})=\zeta$.
\end{proof}

\[
\begin{tikzcd}[column sep=small]
                                                                 & {\mathbf{Flo}_{\tilde{\Sigma}}[X]} \arrow[ldd, "\mathbf{P}"', two heads] \arrow[rdd, "\mathbf{\tilde{P}}"] &         \\
                                                                 &                                                                                                 &         \\
{\mathbf{Flo}_{\tilde{\Sigma}}[X]/\G} \arrow[rr, "\mathbf{E}"'] &                                                                                                 & \flo{Y}
\end{tikzcd}
\]

Observe that the sets of objects in $\mathbf{Flo}_{\tilde{\Sigma}}[X]/G$ and $\flo{Y}$ are the same; first note that the $\G$-action on $X$ and $\mathbf{Loc}_{\tilde{\Sigma}}[X]$ is the same on objects\footnote{Recall that these two categories have the same objects, i.e., elements of the poset $\mathbf{Fac}[X]$.}, so the sets of orbit classes in $\mathbf{Loc}_{\tilde{\Sigma}}[X]/\G$ and $\mathbf{Loc}_{\Sigma}[Y]$ are the same. Secondly, it is clear from the construction of $\tilde{\Sigma}$ that an object $[x]\in\mathbf{Loc}_{\Sigma}[Y]$ is $\Sigma$-critical if and only if $gx$ is $\tilde{\Sigma}$-critical for every $g\in\G$. Hence, when we restrict to the flow categories, the sets of classes of objects are identical. 

Thus, we define $\mathbf{E}:\mathbf{Flo}_{\tilde{\Sigma}}[X]/\G\to\flo{Y}$ to be the identity on objects. Since $\mathbf{P}$ is surjective, for each orbit class of morphisms $[\zeta]\in\mathbf{Flo}_{\tilde{\Sigma}}[X] /\G$ there exists a morphism $\tilde{\zeta}\in\mathbf{Flo}_{\tilde{\Sigma}}[X]$ projecting to $\zeta$. Define $$\mathbf{E}([\zeta])\coloneqq\mathbf{\tilde{P}}(\tilde{\zeta}).$$

Clearly $\E(\text{id}_x)=\text{id}_{\E x}=\text{id}_x$ because the only morphisms in $\mathbf{Flo}_{\tilde{\Sigma}}[X]$ projecting to the identity are identities on $g\lambda(x)$, for all $g\in\G$. Any such map projects to $\text{id}_x$ under $\tilde{\mathbf{P}}$. Also, $\E([\zeta\circ\zeta'])=\E([\zeta])\circ\E([\zeta'])$, and any lifted morphism projecting to $\zeta\circ\zeta'$ can be separated into composable choices of $\tilde{\zeta}$ and $\tilde{\zeta'}$. Whilst choices made separately for $\E([\zeta])\circ\E([\zeta'])$ may not be composable in $\mathbf{Flo}_{\tilde{\Sigma}}[X]$, they will lie in the same respective orbits, and so will project to the same morphisms under $\tilde{\mathbf{P}}$.

\begin{prop}\label{prop:catiso}
The functor $\E:\mathbf{Flo}_{\tilde{\Sigma}}[X]/\G\to\mathbf{Flo}_\Sigma[Y]$ is a well-defined isomorphism of categories.
\end{prop}
\begin{proof}
This follows from the following claims:
\begin{enumerate}
    \item \emph{Claim:} The functor $\E$ is independent of choice of $\tilde{\zeta}$ in $\mathbf{Flo}_{\tilde{\Sigma}}[X]$. \\
    \emph{Proof:} If $\tilde{\zeta}$ and $\tilde{\zeta}^\prime$ are two $\tilde{\Sigma}$-zigzags such that $\mathbf{P}(\tilde{\zeta})=\mathbf{P}(\tilde{\zeta}^\prime)$, then there exists a group element $g\in\G$ such that $g\cdot\tilde{\zeta}=\tilde{\zeta}^\prime$. But ${\tilde{\mathbf{P}}(g\cdot\tilde{\zeta})=\tilde{\mathbf{P}}(\tilde{\zeta})}$ because each forward or backwards pointing arrow comprising $g\cdot\tilde{\zeta}$ and $\tilde{\zeta}$ are in the same orbit class.\label{claim:1} 
    \item \emph{Claim:} If $\zeta$ and $\xi$ are representatives of the same orbit class, then ${\E([\zeta])=\E([\xi])}$. \\
    \emph{Proof:} We can either make the same choice of lifted morphism in $\mathbf{Flo}_{\tilde{\Sigma}}[X]$, or note that by the argument in the proof of Claim~\eqref{claim:1}, two different choices in the same orbit map to the same morphism under $\tilde{\mathbf{P}}$.
    
    \item \emph{Claim:} The functor $\E$ is independent of choice of $\tilde{\Sigma}$-equivalence class of $\tilde{\zeta}$. \\
    \emph{Proof:} Suppose first that $\tilde{\zeta}$ and $\tilde{\zeta}^\prime$ are horizontally equivalent. By definition, they differ by an identity morphisms, which maps to an identity morphism under $\tilde{\mathbf{P}}$. Hence, $\tilde{\mathbf{P}}(\tilde{\zeta})$ is horizontally equivalent to $\tilde{\mathbf{P}}(\tilde{\zeta}^\prime)$ in $\flo{Y}$.

    Suppose instead that $\tilde{\zeta}$ and $\tilde{\zeta}^\prime$ are vertically equivalent, i.e.~they fit into a ladder diagram, which is mapped by $\tilde{\mathbf{P}}$ to a diagram of the same form because $\tilde{\mathbf{P}}$ is an LP-functor (Lemma~\ref{lemma:ptilde}). So $\tilde{\mathbf{P}}(\tilde{\zeta})$ and $\tilde{\mathbf{P}}(\tilde{\zeta}^\prime)$ are vertically equivalent in $\flo{Y}$. 
    
    \item \emph{Claim:} The functor $\E$ is an isomorphism. \\
    \emph{Proof:} We prove this by showing that $\E$ is a bijection on the sets of objects, and $\E$ is an order-preserving bijection on each hom-set. The bijection on objects follows immediately since $\E$ is the identity on objects. $\E$ is injective because:
    \begin{align*}
    \E([\zeta])=\E([\zeta^\prime])&\implies \tilde{\mathbf{P}}(\zeta)=\tilde{\mathbf{P}}(\zeta) \\
    &\implies \text{$\tilde{\zeta}$ and $\tilde{\zeta}^\prime$ are in the same orbit class}\\ &\implies \tilde{\zeta}\sim\tilde{\zeta}^\prime\\ &\implies [\zeta]=[\zeta^\prime].
    \end{align*}
    Surjectivity follows from the fact that both $\mathbf{P}$ and $\tilde{\mathbf{P}}$ are surjective. Similarly, both functors are order-preserving, so $\E$ is also order-preserving.
    \end{enumerate}
    Hence, the categories $\flo{Y}$ and $\textbf{Flo}_{\tilde{\Sigma}}[X]/\G$ are isomorphic.
    \end{proof}
    
Proposition~\ref{prop:catiso} tells us that finding an acyclic partial matching on the quotient and passing to the discrete flow category is the same (up to isomorphism) as finding an equivariant acyclic partial matching in the original space and taking the quotient of this larger flow category.

This gives us the tools to relate the Morse complex of groups over $\flo{X}$ to a complex of groups over the quotient $\mathbf{Flo}_{\tilde{\Sigma}}[X]/\G$. In particular, we will use Lemma~\ref{lem:cogcommute} to find a complex of groups over $\mathbf{Flo}_{\tilde{\Sigma}}[X]/\G$ associated to the $\G$-action on $\mathbf{Flo}_{\tilde{\Sigma}}[X]$ so that its development is isomorphic to that of $\M{}$.

Let $\lambda$ denote the choice of lift function made to construct $\F{}$.

\begin{prop}\label{prop:liftscommute}
There is a choice of lifts and transfers for the canonical complex of groups $\mathbf{M}^\prime$ and monomorphism $\Psi'$ associated to the action on $\mathbf{Flo}_{\tilde{\Sigma}}[X]$, such that the following diagram commutes:

\[
\begin{tikzcd}
{\mathbf{Flo}_{\tilde{\Sigma}}[X]/\G} \arrow[rr, "\E"] \arrow[rdd, "\mathbf{M'}"'] &      & \flo{Y} \arrow[ldd, "\mathbf{M}"] \\
                                                                                         & {}   &                                  \\
                                                                                         & \grp &                                 
\end{tikzcd}
\]
and $\Psi'_\zeta=\Psi_{\E\zeta}$ for every morphism $\zeta\in\mathbf{Flo}_{\tilde{\Sigma}}[X]/\G$.
\end{prop}
\begin{proof}
We define a lift function $\mu$ for $\M{}'$: for any object $[x]$, set $\mu([x])\coloneqq \lambda(\E([x]))$. For any morphism $\zeta$ with representative 
\[
    \begin{tikzcd}
        w \arrow[r, "h_0"] & y_0 & x_0 \arrow[l, "f_0"'] \arrow[r, "h_1"] & y_1 & \cdots \arrow[l, "f_1"'] & x_k \arrow[l, "f_k"'] \arrow[r, "h_{k+1}"] & z,
    \end{tikzcd}
    \]
choose the transfer element $\sigma\coloneqq \sigma(\E\zeta)$, which is (by definition of the Morse complex of groups), $$\sigma(\E\zeta) = \sigma(h_{k+1})\sigma(f_k)^{-1}\cdots \sigma(h_1)\sigma(f_0)^{-1}\sigma(h_0).$$ We need to check that $\sigma$ is a valid choice of transfer element, i.e., that the morphism ${\lambda(\zeta):\lambda(w)\to\sigma^{-1}\lambda(z)}$ is the unique $\tilde{\Sigma}$-zigzag with source $\lambda(w)$ such that $\mathbf{P}(\lambda(\zeta))=\zeta$.

To simplify notation, write 
\begin{align*}
    \tilde{u} &\coloneqq \lambda(u),\ \text{for any object $u$}\\
    \alpha_i &\coloneqq \sigma(h_0)^{-1}\sigma(f_0)\sigma(h_1)^{-1}\cdots \sigma(h_i)^{-1},\ \text{and} \\
    \beta_i &\coloneqq \sigma(h_0)^{-1}\sigma(f_0)\sigma(h_1)^{-1}\cdots \sigma(h_i)^{-1}\sigma(f_{i}).
\end{align*}
Note that $\beta_{k+1} = \sigma^{-1}$. Each $\beta_i$ is such that $\beta_i\tilde{h}_i:\beta_i\tilde{x}_{i-1}\to\alpha_{i+1}\tilde{y}_{i+1}$ is the unique morphism with source $\beta_i\tilde{x}_{i-1}$ projecting to $h_i$ under $\p$. Similarly, each $\alpha_i$ is such that $\alpha_i\tilde{f}_i^{-1}:\alpha_i\tilde{x}_{i-1}\to\beta_i\tilde{y}_i$ is the unique morphism in the orbit of $\tilde{f}_i^{-1}$ with source $\alpha_i\tilde{x}_{i-1}$. This follows from the definition of each transfer element $\sigma(h_i)$ and $\sigma(f_i)$.

Hence, $\tilde{\zeta}$ can be written explicitly as the composition:
\[
\begin{tikzcd}
\tilde{w} \arrow[r, "\tilde{h}_0"] & \alpha_0\tilde{y}_0 \arrow[r, "\alpha_0\tilde{f}_0^{-1}"] & \beta_0\tilde{x}_0 \arrow[r, "\beta_0\tilde{h}_1"] & \alpha_1\tilde{y}_1 \arrow[r, "\alpha_1\tilde{f}_1^{-1}"] & \cdots  \arrow[r, "\alpha_k\tilde{f}_k^{-1}"] & \beta_k\tilde{x}_k \arrow[r, "\beta_k\tilde{h}_k"] & \alpha_{k+1}\tilde{z},
\end{tikzcd}
\]

which is a morphism such that $\mathbf{P}(\tilde{\zeta})=\zeta$, with source $\lambda(w)$ and target $\sigma^{-1}\lambda(z)$, as required. Its uniqueness follows from the uniqueness of each constituent morphism in the composition. By Proposition~\ref{prop:morsecogwd}, this is well-defined on equivalence classes. 

The commutativity of the diagram then follows immediately from this construction, because the 1-morphisms, 2-morphisms and twisting elements for $\mathbf{M}'$ and $\M{}$ are constructed in the same way from the transfer elements, and the maps on objects are the same.

To see that $\Phi_\zeta = \Phi_{\E\zeta}$ for any morphism $\zeta$, note again that it is defined as conjugation by the transfer element corresponding to $\zeta$. But this was chosen precisely so that $\sigma(\zeta)=\sigma(\E\zeta)$. 
\end{proof}

\begin{rem}
    Note that by definition of the induced action on $\mathbf{Flo}_{\tilde{\Sigma}}[X]$, the localisation and inclusion functors $\mathbf{L}_{\tilde{\Sigma}}$ and $\mathbf{J}_{\tilde{\Sigma}}$ are both $\G$-equivariant.
\end{rem}

\begin{proof}[Proof of Theorem~\ref{thm:homequiv}]
Given $\F{}$ and $\Phi$, we can construct the Morse complex of groups $\M{}$ and associated monomorphism $\Psi$, as described in Section~\ref{sec:morsecog}. Since $\Sigma$ is $\F{}$-compatible, we can build an equivariant acyclic partial matching $\tilde{\Sigma}$ on $X$ (Lemma~\ref{lem:sigmatilde}). Furthermore, the $\G$-action on $X$ pushes to an action on $\mathbf{Flo}_{\tilde{\Sigma}}[X]$  (Lemma~\ref{lem:floaction}), and the quotient $\mathbf{Flo}_{\tilde{\Sigma}}[X]/\G$ is isomorphic to $\flo{Y}$ (Proposition~\ref{prop:catiso}). There is a choice of lifts and transfers to build the canonical complex of groups $\mathbf{M}^\prime$ and morphism $\Psi^\prime$ associated to the action on $\mathbf{Flo}_{\tilde{\Sigma}}[X]$ making the diagram
\[
\begin{tikzcd}
{\mathbf{Flo}_{\tilde{\Sigma}}[X]/\G} \arrow[rr, "\E"] \arrow[rdd, "\mathbf{\tilde{M}}"'] &      & \flo{Y} \arrow[ldd, "\mathbf{M}"] \\
                                                                                         & {}   &                                  \\
                                                                                         & \grp &                                 
\end{tikzcd}
\]
commute (Proposition~\ref{prop:liftscommute}). Hence, ${\D(\M{},\Psi)\cong\D(\M{}^\prime,\Psi^\prime)}$ as categories (Lemma~\ref{lem:cogcommute}). By the Basic Construction (Proposition~\ref{prop:bascon}), $\D(\M{}^\prime,\Psi^\prime)$ is $\G$-equivariantly isomorphic to $\mathbf{Flo}_{\tilde{\Sigma}}[X]$, and the latter has classifying space homotopy-equivalent to $X$ (Theorem~\ref{thm:flocat}). Since the localisation and inclusion functors, $\mathbf{L}_{\tilde{\Sigma}}$ and $\mathbf{J}_{\tilde{\Sigma}}$, are $\G$-equivariant with respect to the induced $\G$-action on $\mathbf{Loc}_{\tilde{\Sigma}}[X]$ and $\mathbf{Flo}_{\tilde{\Sigma}}[X]$, the classifying space $|\Delta\D(\M{},\Psi)|$ of the development of the Morse complex of groups and its associated monomorphism is $\G$-equivariantly homotopy-equivalent to $X$, as required.
\end{proof}

	\section{An Example}\label{sec:example}
	
	We present a detailed example to demonstrate the entire process, from a regular action on a simplicial complex to the higher development of the associated discrete Morse complex about an acyclic partial matching on the orbit space. \\

Let $X$ be the triangulation of an annulus, as show in Figure \ref{fig:action}, equipped with the (regular) left action of $D_8=\langle r,s\ |\ r^4=s^2=e,r^3s=sr\rangle$, where $r$ is the anticlockwise rotation by $\pi/2$ and $s$ is a reflection in the $y$-axis. The orbit space $Y$ in this case is a fundamental domain, \emph{i.e.}~a subcomplex of $X$; it is a simplicial complex with four 0-simplices $\{v_0,...,v_3\}$, five 1-simplices $\{u_0,...,u_4\}$ and two 2-simplices $\{w_0,w_1\}$. \\

\begin{figure}[ht]
    \centering
    \fontsize{9pt}{9pt}
    \def\svgwidth{\columnwidth}
    \scalebox{1}{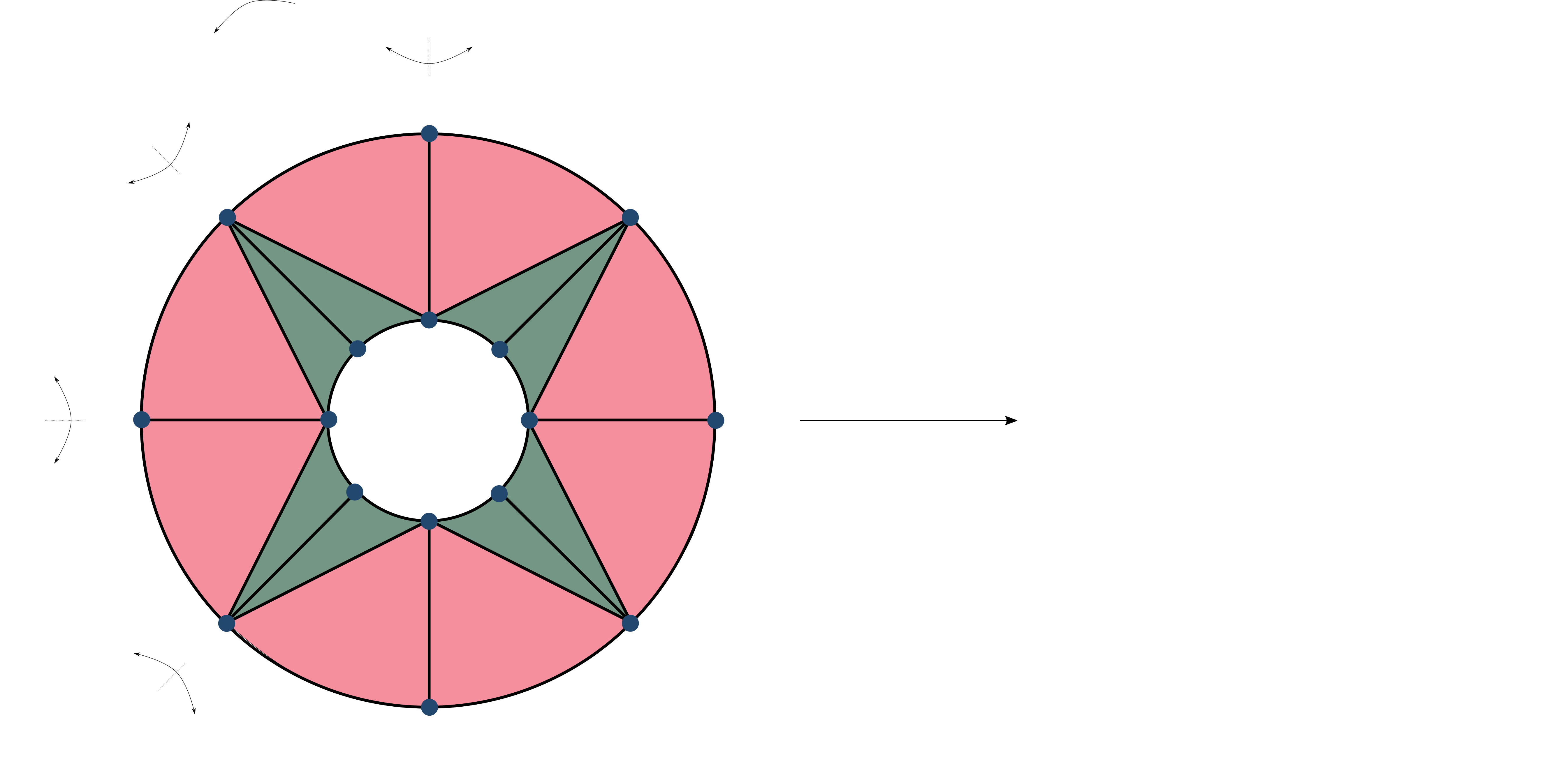}
    \caption{On the left, a triangulation $X$ of the annulus equipped with a regular simplicial (left) group action of $D_8={\langle r,s\ |\ r^4=s^2=e,r^3s=sr\rangle}$, where $r$ is the anticlockwise rotation by $\pi/2$ and $s$ is the reflection in the $y$-axis. The two orbits of the $2$-simplices are coloured in pink and green, respectively. On the right, the orbit space $Y$, a simplicial complex with four 0-simplices $\{v_0,...,v_3\}$, five 1-simplices $\{u_0,...,u_4\}$ and two 2-simplices $\{w_0,w_1\}$.}
    \label{fig:action}
\end{figure}

To construct the complex of groups associated to this action, as well as the morphism of complexes of groups, we make a choice of lifts and transfers; for each $y\in Y$ we choose a simplex $\overline{y}\in X$ such that $p(\overline{y})=y$, and these choices are highlighted on the left of Figure~\ref{fig:cog}. For each pair $x\to y$ of simplices related by a face relation in $Y$ (\emph{i.e.~}for each pair of simplices related by a morphism in $\fac{Y}$) we choose a group element $g\in\G$ such that $g^{-1}\cdot\overline{y}$ is a face of $\overline{x}$ in $X$. For example, $\overline{w_1}$ and $\overline{u_4}$ are adjacent in $X$ so their transfer is trivial; however, $r^2s\cdot\overline{u_1}$ is a face of $\overline{w_0}$ (whereas $\overline{u_1}$ is not), so in the complex of groups, the morphism between the stabilisers $\F{}_{w_0} \to \F{}_{u_1}$ is conjugation by $(r^2s)^{-1}=r^2s$. After making all of these choices, we obtain a complex of groups over $Y$, as illustrated in Figure~\ref{fig:cog}. The morphism $\Phi:\F{}\to\cst$ follows immediately from the choice of $\overline{y}$ for each $y\in Y$ and transfers.\\

\begin{figure}[ht]
    \centering
    \fontsize{8pt}{8pt}
    \def\svgwidth{\columnwidth}
    \scalebox{1}{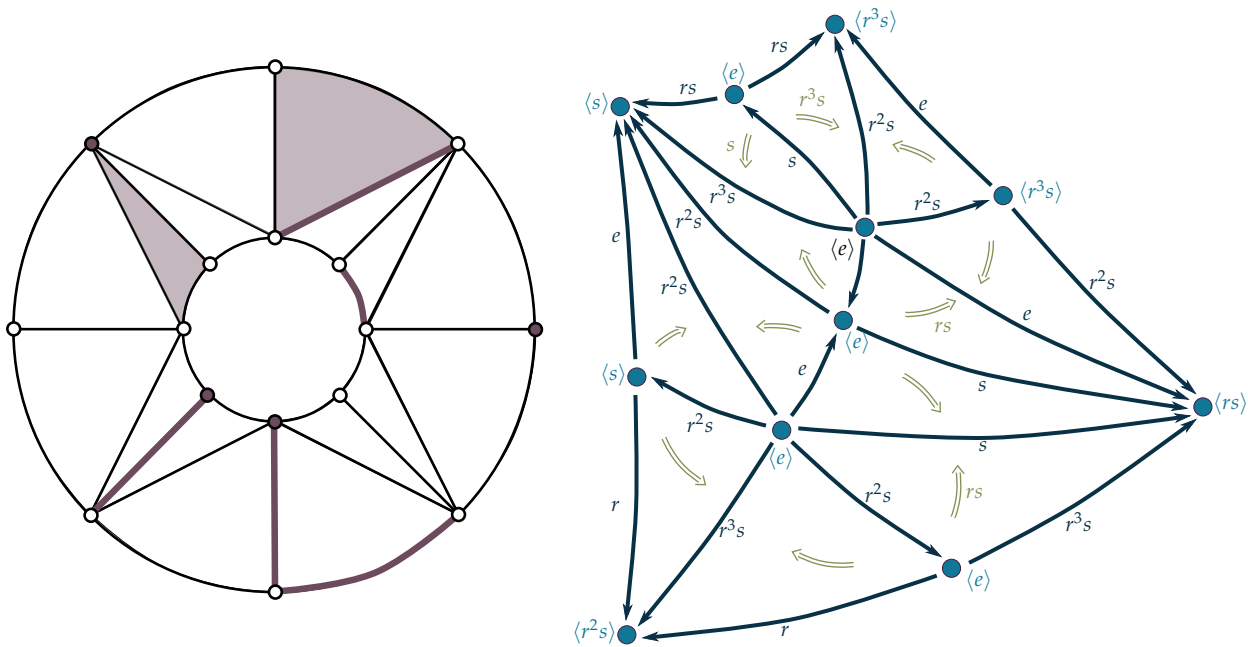}
    \caption{\textbf{The complex of groups associated to a group action.} On the left, a choice of lifts for each simplex in the orbit space $Y$. On the right, the complex of groups associated to this group action and choice of lifts and transfers. The transfers determine the 1-morphisms, which are decorated on the right by the group element by which the morphism conjugate. Only non-trivial 2-morphisms are labelled.}
    \label{fig:cog}
\end{figure}

Equipped with a complex of groups and a morphism into $\cst$, we have enough data to use Proposition~\ref{prop:bascon} to build a simplicial complex that is $\G$-equivariantly isomorphic to $X$. However, we are looking to reduce the size of $Y$ by finding an acyclic partial matching on $Y$ before building a complex of groups to reconstruct. We choose the smallest non-trivial matching consisting of a single pairing $\Sigma=\{f:u_4\rightarrow w_1 \}$. This leads to the discrete flow category displayed on the right hand side of Figure \ref{fig:flowcat}. The lifts of $u_4$ and $w_1$ are adjacent in $X$, yielding a compatible matching. On the right, the discrete flow category $\flo{Y}$; there are no objects corresponding to cells $u_4$ or $w_1$ as these are not critical. Each morphism is an equivalence class of zigzags under the relations given in Definition~\ref{def:catloc}, with representative zigzags decorating the morphisms above. There is a (non-trivial) poset of morphisms for each of the pairs of objects $(w_0,v_2)$ and $(w_0,v_0)$. \\

\begin{figure}[ht!]
    \centering
    \fontsize{8pt}{8pt}
    \def\svgwidth{\columnwidth}
    \scalebox{1}{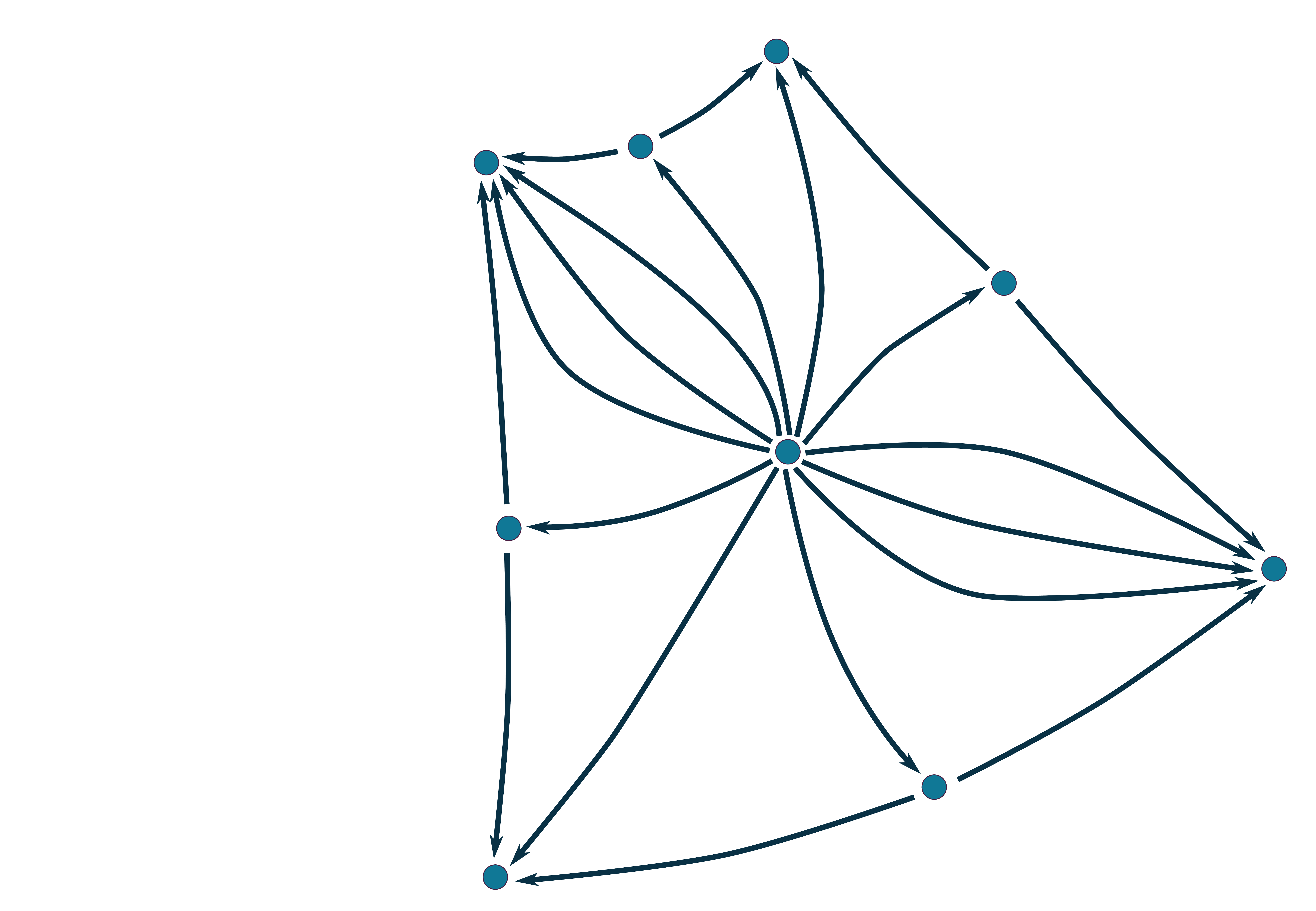}
    \caption{\textbf{The discrete flow category.}}
    \label{fig:flowcat}
\end{figure}

The discrete flow category is the domain of our Morse complex of groups $\M{}$, which is constructed in full at the top of Figure~\ref{fig:morsecog}. At the bottom, we see part of its higher development, whose classifying space is $\G$-equivariantly homotopy equivalent to $X$.

\begin{figure}[ht!]
    \centering
    \fontsize{8pt}{8pt}
    \def\svgwidth{\columnwidth}
    \scalebox{1.1}{\input{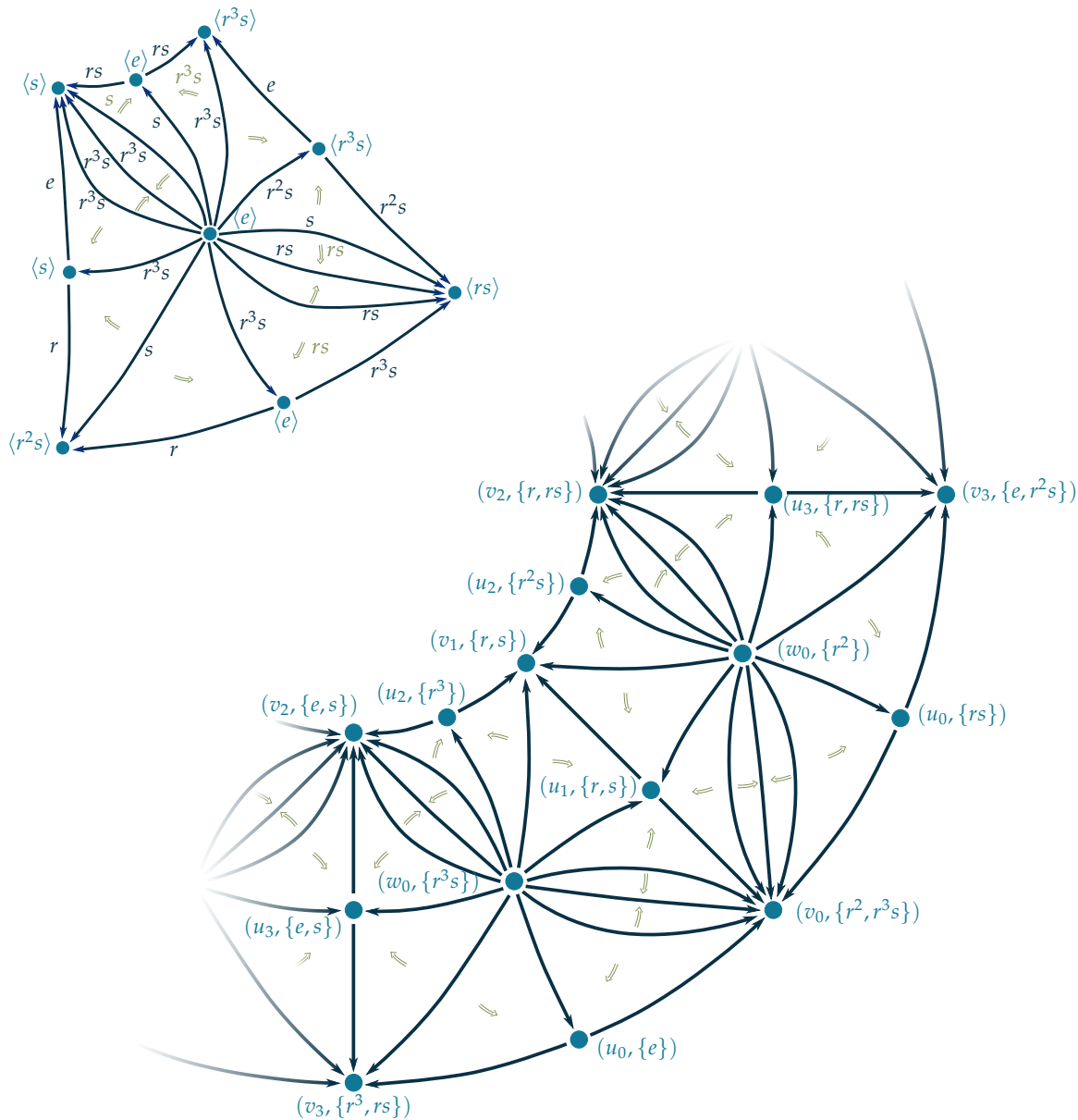}}
    \caption{\textbf{The discrete Morse complex and its higher development.} The discrete Morse complex (top) is a complex of groups over the discrete flow category $\flo{Y}$. All 1-morphisms are conjugation by the group element decorating the arrow. All non-trivial 2-morphisms are labelled. Part of its higher development (bottom). The classifying space of the higher development is $\G$-equivariantly homotopy equivalent to $X$, i.e., it has the homotopy-type of $S^1$.}
    \label{fig:morsecog}
\end{figure}

	\bibliographystyle{abbrv}

\end{document}